\documentclass[11pt,a4paper]{amsart}

\usepackage{a4wide}
\usepackage{amsfonts,amsthm,amsmath,amssymb,amscd}
\usepackage{graphicx,color,import}

\theoremstyle{plain} 
\newtheorem{lem}{Lemma}[section]

\newtheorem{thm}[lem]{Theorem}
\newtheorem{cor}[lem]{Corollary}
\newtheorem*{cor*}{Corollary}

\newtheorem*{corA'}{Corollary A'}
\newtheorem*{corB'}{Corollary B'}

\theoremstyle{definition}
\newtheorem{defn}[lem]{Definition}
\newtheorem*{defn*}{Definition}
\newtheorem{ex}[lem]{Example}
\newtheorem*{ex*}{Example}
\newtheorem{rem}[lem]{Remark}
\newtheorem*{rem*}{Remark}

\newtheorem*{ass}{Assumptions}
\newtheorem*{quest}{Question}

\theoremstyle{remark}

\DeclareMathOperator{\diam}{diam}
\DeclareMathOperator{\dist}{dist}

\DeclareMathOperator{\Arg}{Arg}

\DeclareMathOperator{\supp}{supp}

\newcommand{\C}{\mathbb C}
\newcommand{\R}{\mathbb R}
\newcommand{\Z}{\mathbb Z}
\newcommand{\N}{\mathbb N}

\newcommand{\D}{\mathbb D}

\newcommand{\BB}{\mathcal B}

\newcommand{\AAA}{\mathcal A}

\newcommand{\KK}{\mathcal K}

\newcommand{\EE}{\mathcal E}
\newcommand{\FF}{\mathcal F}

\newcommand{\II}{\mathcal I}

\newcommand{\ul}{{\underline{\lambda}}}
\renewcommand{\aa}{{\underline{a}}}
\newcommand{\bb}{{\underline{b}}}
\newcommand{\QQ}{\mathcal Q}

\newcommand{\bd}{\partial}

\renewcommand{\Re}{\textup{Re}}
\renewcommand{\Im}{\textup{Im}}

\begin{document}

\title[On the dimension of points which escape to infinity at given rate]{On the dimension of points which escape to infinity at given rate under exponential iteration} 
\date{\today}

\author{Krzysztof Bara\'nski}
\address{Institute of Mathematics, University of Warsaw,
ul.~Banacha~2, 02-097 Warszawa, Poland}
\email{baranski@mimuw.edu.pl}

\author{Bogus{\l}awa Karpi\'nska}
\address{Faculty of Mathematics and Information Science, Warsaw
University of Technology, ul.~Ko\-szy\-ko\-wa~75, 00-661 Warszawa, Poland}
\email{bkarpin@mini.pw.edu.pl}

\thanks{Supported by the National Science
Centre, Poland, grant no 2018/31/B/ST1/02495.}

\bibliographystyle{alpha}

\begin{abstract} We prove a number of results concerning the Hausdorff and packing dimension of sets of points which escape (at least in average) to infinity  at a given rate under non-autonomous iteration of exponential maps. In particular, we generalize the results proved by Sixsmith in 2016 and answer his question on annular itineraries for exponential maps.
\end{abstract}

\maketitle

\section{Introduction}\label{sec:intro}

In this paper we study the iteration of \emph{exponential maps}
\[
E_\lambda(z) = \lambda e^z, \qquad z \in \C, \quad \lambda \in \C \setminus \{0\}
\]
and, more generally, the \emph{non-autonomous exponential iteration}
\[
\cdots \circ E_{\lambda_n} \circ \cdots \circ E_{\lambda_1},
\]
where $\lambda_1, \lambda_2, \ldots \in \C \setminus \{0\}$. We study the dimension of sets of points $z \in \C$  which escape to infinity (at least in average) at a prescribed speed, meaning that $a_n \le |E_{\lambda_n} \circ \cdots \circ E_{\lambda_1}(z)| \le b_n$ for  given sequences $a_n, b_n$.

For a transcendental entire map $f\colon \C \to \C$ the \emph{escaping set} $I(f)$ is defined as
\[
I(f) = \{z \in \C: |f^n(z)| \to \infty \text{ as } n \to \infty\},
\]
while the \emph{Julia set} $J(f)$ is the set of points $z \in \C$, where the iterates $f^n$ do not form a normal family in any neighbourhood of $z$. There is a close relationship between the Julia set and escaping set -- the set $J(f)$ is equal to the boundary of $I(f)$, as proved by Eremenko in  \cite{ere-escaping}. Furthermore, Eremenko and Lyubich showed in~\cite{erem-lyub} that for functions $f$ in the class 
\[
\BB = \{\text{transcendental entire maps with a bounded set of critical and asymptotic values}\}, 
\]
in particular for exponential maps, the escaping set is contained in the Julia set, so $J(f) = \overline{I(f)}$. 

The dimension of the Julia sets of transcendental entire functions was first considered by McMullen in \cite{mcmullen-exp}, who proved that all Julia sets of exponential maps have Hausdorff dimension ($\dim_H$) equal to $2$. Since then, the question of the size of the Julia and escaping sets and their dynamically defined subsets has attracted a lot of attention (see, among others, the references mentioned in this section).

In fact, in \cite{mcmullen-exp} it was showed that $\dim_H (A(E_\lambda)) = 2$, where 
\[
A(f) = \{z \in I(f): |f^{n+l}(z)| \ge M^n_f(R), n \in \N, \text{ for some } l \ge 0\}
\]
is the \emph{fast escaping set} of $f$, introduced by Bergweiler and Hinkkanen in \cite{BergHink} and then studied by Rippon and Stallard in \cite{RS-fast}. Here $R > 0$ is a large fixed number, $M_f(r) = \max_{|z| = r} |f(z)|$ for $r > 0$ and $M_f^n$ denotes the $n$-th iterate of $M_f(\cdot)$. In fact, results by Bergweiler, Karpi\'nska and Stallard \cite{BKS} and Rippon and Stallard \cite{RS-quitefast} imply that $\dim_H A(f) = 2$ for all transcendental entire $f \in \mathcal B$ of finite order or `not too large' infinite order. It is then a natural question to determine the dimension of subsets of $J(f) \cap I(f)$ consisting of points escaping to infinity at a slower rate, or other dynamically defined subsets of $J(f) \cap I(f) \setminus A(f)$. A number of such sets, including 
\emph{slow escaping set}
\[
L(f) = \Big\{z \in I(f) : \limsup_{n\to\infty} \frac 1 n \log |f^n(z)| < \infty\Big\} 
\]
and \emph{moderately slow escaping set}
\[
M(f) = \Big\{z \in I(f) : \limsup_{n\to\infty} \frac 1 n \log\log |f^n(z)| < \infty\Big\} 
\]
have been defined and studied in recent years (see e.g.~\cite{RS-quitefast, RS-slow}).

\begin{rem}[{\bf Topological structure}{}] It is well-known (see e.g.~\cite{DK,DT,aarts,SZ}) that escaping sets of exponential maps contain disjoint hairs (simple curves converging to $\infty$ with some special properties). For exponential maps with an attracting fixed point and, more generally, for maps of finite order from the class $\BB$ with a unique Fatou component, the Julia set is the union of hairs together with their endpoints (see~\cite{B-paradox,hairs,RRRS}). In \cite{RempeStallard-fast} Rempe, Rippon and Stallard showed that for all transcendental entire $f \in \mathcal B$ of finite order, the hairs without endpoints are contained in $A(f)$. Therefore, for exponential maps with an attracting fixed point, the set $I(E_\lambda) \setminus A(E_\lambda)$ is contained in the union of endpoints of the hairs. 
\end{rem}

\begin{rem}[{\bf Points with bounded trajectories}{}] Let $J_{bd}(f)$ denote the set of points in the Julia set of $f$ with bounded trajectories. In \cite{B-area} it is proved that the Hausdorff dimension of $J_{bd}(E_\lambda)$ is larger than $1$. Furthermore, in \cite{UZ-finer} it is showed that $\dim_H(J(E_\lambda) \setminus I(E_\lambda)) \in (1,2)$ for all hyperbolic exponential maps $E_\lambda$.
More generally, $\dim_H(J_{bd}(f)) > 1$ for every transcendental entire map in the class $\BB$ (see \cite{BKZ-tracts}) and $\dim_H(J(f) \setminus (I(f) \cup J_{bd}(f)) > 1$ for every transcendental entire map $f$ in the class $\BB$ (see \cite{OS}).
\end{rem}

To conduct a refined analysis of the sets of points with given escape rate, for a transcendental entire map $f$ and sequences $\aa = (a_n)_{n=1}^\infty$, $\bb = (b_n)_{n=1}^\infty$ with $0 < a_n \le b_n$ let
\begin{align*}
I_\aa^\bb(f) &= \{z \in \C: a_n \le |f^n(z)| \le b_n \text{ for every sufficiently large } n \in \N\}\\
I^\bb(f) &= \{z \in \C: |f^n(z)| \le b_n \text{ for every sufficiently large } n \in \N\}.
\end{align*}
To guarantee that the sets $I_\aa^\bb(f)$ are not empty, one usually assumes that the sequence $\aa$ is \emph{admissible}, which roughly means $a_{n+1} < M_f(a_{n})$ (with a precise definition depending on the context). 

Surprisingly, a natural question of determining the dimension of the sets $I_\aa^\bb(f)$ has not been answered completely even for the well-known exponential family. Let us summarize what is known about the size of the sets $I_\aa^\bb(E_\lambda)$ and, more generally, the sets $I_\aa^\bb(f)$ for $f \in \BB$. In \cite{Rempe-exp} Rempe proved that $I_\aa^\bb(E_\lambda) \neq \emptyset$ for every admissible sequence $\aa = (a_n)_{n=1}^\infty$ with $a_n \to\infty$ and $b_n = ca_n$, $c > 1$. The result was generalized by Rippon and Stallard in \cite{RS-slow} to the case of arbitrary transcendental entire (or meromorphic) maps $f$. Moreover, they showed that if $b_n \to \infty$, then $I^{\bb}(f) \cap I(f) \neq \emptyset$. In \cite{BP} Bergweiler and Peter proved that $\dim_H (I(f) \cap I^{\bb}(f)) \ge 1$ for every transcendental entire map $f$ in the class $\BB$, provided $b_n \to \infty$. 

In \cite{KU} Karpi\'nska and Urba\'nski, considering a related topic, studied the Hausdorff dimension of subsets of the escaping set for exponential maps consisting of points whose symbolic itineraries (describing of the imaginary part of $E_\lambda^n(z)$) grow in modulus to infinity at a given rate. They found that the Hausdorff dimension of these sets can achieve any number in the interval $[1,2]$. As noted in \cite{sixsmith}, the subsets of $I(E_\lambda)$ considered in \cite{KU} are contained in the fast escaping set $A(E_\lambda)$.

A motivation for our work was the paper \cite{sixsmith} by Sixsmith, who proved several results on the dimension of the sets $I_\aa^\bb(E_\lambda)$. In particular, he showed that $\dim_H I_\aa^\bb(E_\lambda) = 1$ in the following cases:
\begin{itemize}
\item[(a)] $a_n = c_1 R^n$, $b_n = c_2 R^n$ for $c_1, c_2 > 0$, $R > 1$,
\item[(b)] $a_n = n^{(\log^+)^p(n)}$, $b_n = R^n$ for $p \in \N$, $R > 1$, where $(\log^+)^p$ denotes the $p$-th iterate of $\log^+ = \max(\log, 0)$,
\item[(c)] $a_n = e^{n {\log^+}^p(n)}$, $b_n = e^{e^{pn}}$ for $p \in \N$,
\item[(d)] $\lim_{n\to\infty}a_n = \infty$, $a_{n+1} < R^{a_n^{(1/\log R)}}$, $\lim_{n\to\infty}\frac{\log a_{n+1}}{\log(a_1 \cdots a_n)} = 0$ and $b_n = Ra_n$ for large $n$, where $R > 1$ is a sufficiently large constant. 
\end{itemize}
Note that in the cases (a)--(b) the sets $I_\aa^\bb(E_\lambda)$ are contained in the slow escaping set $L(E_\lambda)$, while in the cases (c)--(d) they are subsets of the moderately slow escaping set $M(E_\lambda)$.
 
In \cite[Remark 2]{sixsmith} the author stated a question, whether the condition in (d) can be weakened. In this paper we answer this question, extending the results described in (a)--(d) and proving a number of facts concerning the dimension of points with given escape rate. The results are presented in a more general settings of non-autonomous iteration 
\[
E_\ul = (E_{\lambda_n} \circ \cdots \circ E_{\lambda_1})_{n=1}^\infty
\]
of exponential maps, with an arbitrary choice of $\lambda_n \in \C \setminus \{0\}$.
Furthermore, the points under consideration are not necessarily escaping. Generally, we only assume that $(a_n)_{n=1}^\infty$ is admissible, $a_n > a$ for large $a$, $(a_1 \cdots a_n)^{1/n} \to \infty$ as $n \to \infty$,  and $b_n \ge c a_n$ for $c > 1$.

Let us summarize the main results of the paper. The exact formulations are contained in Section~\ref{sec:results}.

\begin{itemize}
\item 
In Theorem~\ref{thm:A} we present a general condition which implies that the Hausdorff dimension of $I_\aa^\bb(E_\ul)$ is at most $1$.

\item
In Theorem~\ref{thm:B} and Corollary~\ref{cor:dim} we provide basic estimates for the Hausdorff and packing dimension of $I_\aa^\bb(E_\ul)$ in terms of the growth of moduli of the annuli $\{z\in \C: \{a_n \le |z| \le b_n\}$  compared to the mean geometric growth of the sequences $(a_n)_{n=1}^\infty$, $(b_n)_{n=1}^\infty$.

\item
Corollary~\ref{cor:1-2} provides conditions under which the dimensions achieve extremal values $1$ or $2$.

\item
In Theorem~\ref{thm:any_a}, generalizing the results of \cite{sixsmith}, we show that the sets $I_\aa^\bb(E_\ul)$ with moderately slow escape rate have Hausdorff dimension $1$.

\item Theorem~\ref{thm:rate} shows the same for the sets of points with any given exact growth rate.

\item In Theorem~\ref{thm:thin} we provide exact formulas for the Hausdorff and packing dimension of $I_\aa^\bb(E_\ul)$  in the case $\sup_n |\lambda_n| < \infty$, $\frac{\log b_n}{\log a_n} \to 1$.
 
\item
In Theorem~\ref{thm:examples} we show that the packing dimension of $I_\aa^\bb(E_\ul)$ can achieve any value in the interval $[1,2]$, with the Hausdorff dimension being equal to $1$. 
\end{itemize}

At the end of Section~\ref{sec:results} we state a question, which we find interesting, whether there exists a set $I_\aa^\bb(E_\ul)$ with the Hausdorff dimension between $1$ and $2$.

In \cite{sixsmith}, the mentioned result~(d) was described in the language of \emph{annular itineraries}, which are the sequences of non-negative integers $s_n$ defined by the condition $f^n(z) \in \AAA_{s_n}$ for a partition of the plane by a sequence of concentric annuli $\AAA_s$, $s \ge 0$, with radii growing to infinity as $s \to \infty$. In \cite{RS-annular} Rippon and Stallard proved that for all transcendental entire maps $f$ there exist escaping points with any given admissible annular itinerary.
In Section~\ref{sec:annular} of this paper we also take up this approach, determining the dimension of sets of points sharing given annular itinerary under non-autonomous exponential iteration for various sequences of the annuli $\AAA_s$ (Theorems~\ref{thm:annular} and~\ref{thm:annular2}).

\section{Results}\label{sec:results}

\subsection{Preliminaries}\label{subsec:prelim}

We consider a non-autonomous exponential iteration 
\[
E_\ul = (E_{\lambda_n} \circ \cdots \circ E_{\lambda_1})_{n=1}^\infty
\]
for $\ul = (\lambda_n)_{n=1}^\infty$, where $\lambda_n \in \C \setminus \{0\}$. 
We extend the definition of the sets $I_\aa^\bb(f)$ to the non-autonomous setup, setting
\[
I_\aa^\bb(E_\ul) = \{z \in \C: a_n \le |E_{\lambda_n} \circ \cdots \circ E_{\lambda_1}(z)| \le b_n \text{ for every sufficiently large } n \in \N\}.
\]
for $\aa = (a_n)_{n=1}^\infty$, $\bb = (b_n)_{n=1}^\infty$ with $0 < a_n < b_n$. Note that, in general, the sequences $a_n$ and $b_n$ need not be increasing and need not tend to infinity. We denote by $\Delta_n$ the (suitably normalized) modulus of the annulus $\{z \in \C: a_n \le |z| \le b_n\}$, i.e.
\[
\Delta_n = \log \frac{b_n}{a_n}.
\]
Our results concern the Hausdorff and packing dimension (see e.g.~\cite{falconer,mattila} for definitions), which are denoted, respectively, by $\dim_H$ and $\dim_P$. Recall that 
\[
\dim_H \le \dim_P.
\]

\subsection{General estimates}\label{subsec:estimates}

Our first result provides an upper estimate of the Hausdorff dimension
of the sets $I_\aa^\bb(E_\ul)$. Geometrically, it states that $\dim_H I_\aa^\bb(E_\ul)$ can be larger than $1$ only if the moduli $\Delta_n$ grow quickly enough compared to the mean geometric growth of the sequence $\aa$.
The proof is contained in Section~\ref{sec:proof-A}.

\begin{thm} \label{thm:A} Let $\aa = (a_n)_{n=1}^\infty$, $\bb = (b_n)_{n=1}^\infty$ be such that $\inf_{n\in\N} a_n > 0$ and
\begin{equation}\label{eq:assA}
\liminf_{n\to\infty}\left(\frac{\Delta_{n+1}}{a_1 \cdots a_n}\right)^{\frac 1 n} =0.
\end{equation}
Then $\dim_H I_\aa^\bb(E_\ul) \leq  1$.
\end{thm}

\begin{rem}\label{rem:A} 
It is straightforward to check that \eqref{eq:assA} holds provided
\[
\lim_{n\to\infty} (a_1 \cdots a_n)^{1/n} = \infty
\quad\text{and} \quad
\liminf_{n\to\infty} \frac{\log \Delta_{n+1}}{\log(a_1 \cdots a_n)}<1
\]
or
\[
\limsup_{n\to\infty} (a_1 \cdots a_n)^{1/n} = \infty
\quad\text{and} \quad
\limsup_{n\to\infty} \frac{\log \Delta_{n+1}}{\log(a_1 \cdots a_n)}<1.
\]
\end{rem}

Before formulating next results, we introduce the notion of admissibility used in our context. Recall that this condition, bounding the growth of the sequences $(a_n)_{n=1}^\infty$, $(b_n)_{n=1}^\infty$, is introduced to ensure that the sets $I_\aa^\bb(E_\ul)$ under consideration are non-empty. 

\begin{defn}\label{defn:admis}
We say that sequences $\aa= (a_n)_{n=1}^\infty$, $\bb = (b_n)_{n=1}^\infty$ are \emph{admissible}, if for sufficiently large $n$ we have 
\[
a_{n+1} \le |\lambda_{n+1}|e^{q a_n}, \qquad b_{n+1} \ge|\lambda_{n+1}| e^{-q a_n}
\]
for a constant $0 < q < 1$. 
\end{defn}

From now on, our general assumptions will be the following.

\begin{ass}

\

\begin{itemize}
\item[(a)]  The sequences $\aa = (a_n)_{n=1}^\infty, \bb = (b_n)_{n=1}^\infty$ are admissible.
\item[(b)] $(a_1 \cdots a_n)^{1/n} \to \infty$ as $n \to \infty$.
\item[(c)] $\Delta_n > \Delta > 0$ for $n \in \N$.
\item[(d)] $\liminf_{n \to \infty} a_n > a$, where $a$ is a sufficiently large constant, depending on $\Delta$ and $q$ from Definition~\ref{defn:admis}. 
\end{itemize}
\end{ass}
\begin{rem}

Note that if $a_n \to \infty$ and the sequence $|\lambda_n|$ is bounded away from $0$ and $\infty$, then the assumptions reduce to $\Delta_n > \Delta > 0$ and $a_{n+1} \le e^{q a_n}$, $0 < q < 1$, for large $n$. 
\end{rem}

The next result provides general lower and upper estimates of the Hausdorff and packing dimension of the sets $I_\aa^\bb(E_\ul)$ in terms of the growth of the moduli $\Delta_n$  compared to the growth of the sequences $\aa$ and $\bb$. The proof is contained in Sections~\ref{sec:proof-B-prelim}--\ref{sec:proof-B-below}.

\begin{thm}\label{thm:B} Suppose that the assumptions $($a$)$--$($d$)$ are satisfied. Then
\begin{alignat*}{3}
&1 + \inf_x\liminf_{n \to \infty}\phi_n(x) &\le \dim_H I_\aa^\bb(E_\ul) &\le 1 + \sup_x\liminf_{n \to \infty}\phi_n(x),\\
&1 + \inf_x \limsup_{n \to \infty}\psi_n(x) &\le \dim_P I_\aa^\bb(E_\ul) &\le 1 + \sup_x \limsup_{n \to \infty}\psi_n(x),
\end{alignat*}
where $x = (x_1, x_2 , \ldots) \in [a_1, b_1] \times [a_2, b_2] \times \cdots$ and
\begin{align*}
\phi_n(x) &= \frac{\log\big(\min(\Delta_2, x_1) \cdots \min(\Delta_n, x_{n-1})\big)}{\log(x_1 \cdots x_n) - \log\min(\Delta_{n+1}, x_n)},\\
\psi_n(x) &= \frac{\log\big( \min(\Delta_2, x_1) \cdots \min(\Delta_{n+1}, x_n)\big)}{\log(x_1  \cdots x_n)}.
\end{align*}
\end{thm}

Theorem~\ref{thm:A} and \ref{thm:B}  
imply a number of corollaries, presented below. The first one shows, among others, that the Hausdorff dimension of the considered sets $I_\aa^\bb(E_\ul)$ is at least $1$.

\begin{cor}\label{cor:dim} Under the assumptions of Theorem~{\rm\ref{thm:B}}, 
\begin{align*}
1 &\le \dim_H I_\aa^\bb(E_\ul) \le 1 + \liminf_{n\to\infty}\frac{\log\big(\Delta_1\cdots  \Delta_n\big)}{\log(a_1\cdots a_{n-1}) +\log^+ (a_n/\Delta_{n+1})},\\
1 &\le \dim_P I_\aa^\bb(E_\ul) \le 1 + \limsup_{n\to\infty}\frac{\log\big(\Delta_1\cdots  \Delta_{n+1}\big)}{\log(a_1\cdots a_n)}.
\end{align*}
If, additionally, 
\begin{equation}\label{eq:dimm}
\sup_{n\in\N} \frac{\Delta_{n+1}}{a_n} < \infty,
\end{equation}
then
\[
\dim_H I_\aa^\bb(E_\ul) = 1,\qquad
\dim_P I_\aa^\bb(E_\ul) \ge 1 + \limsup_{n\to\infty}\frac{\log\big(\Delta_1\cdots \Delta_{n+1}\big)}{\log(b_1\cdots b_n)}.
\]
\end{cor}

\begin{proof} Note first that by assumptions, $\log(a_1 \cdots a_{n-1}) > 0$ and 
$\min(\Delta_n, a_{n-1}) \ge c$ for large $n$ and some constant $c > 0$. Hence, the numerator in the expression for $\phi_n$ in Theorem~\ref{thm:B} is larger than $C n$ for a constant $C \in \R$, while the denominator is not smaller than $\log(a_1 \cdots a_{n-1})$, which is positive. Thus, 
\[
\phi_n \ge -\frac{|C|n}{\log(a_1 \cdots a_{n-1})} \xrightarrow[n \to \infty]{} 0
\]
since $(a_1 \cdots a_n)^{1/n} \to \infty$. Hence, $\liminf_{n\to\infty} \phi_n \ge 0$, so $\dim_H I_\aa^\bb(E_\ul) \ge 1$. The remaining assertions follow from Theorem~\ref{thm:A}, Remark~\ref{rem:A} and Theorem~\ref{thm:B} in a straightforward way. 
\end{proof}

\begin{rem}\label{rem:dim} If $\sup_{n\in\N} |\lambda_n| < \infty$, then the condition \eqref{eq:dimm} holds provided $\sup_{n\in\N} \frac{\log b_n}{\log a_n} < \infty$.
\end{rem}
\begin{proof} If $\sup_{n\in\N} \frac{\log b_n}{\log a_n} < \infty$, then $\Delta_{n+1} \le c \log a_{n+1}$ for a constant $c > 0$. This together with the admissibility implies 
\[
\frac{\Delta_{n+1}}{a_n} \le c \frac{\log a_{n+1}}{a_n} \le c \Big(q + \frac{\log|\lambda_{n+1}|}{a_n}\Big) \le c \Big(q + \frac{\log^+\sup_n|\lambda_n|}{a}\Big)
\]
for large $n$.
\end{proof}

The following fact provides conditions under which the Hausdorff and packing dimension of $I_\aa^\bb(E_\ul)$ achieve extremal values $1$ or $2$. Note that the assertion (d) is a refinement of the McMullen result from \cite{mcmullen-exp}.

\begin{cor}\label{cor:1-2} Under the assumptions of Theorem~{\rm\ref{thm:B}}, 
\begin{itemize}

\item[(a)] If ${\displaystyle\limsup_{n\to\infty}} \frac{\log\Delta_{n+1}}{\log a_n} \leq 0$, then $\dim_H I_\aa^\bb(E_\ul) = \dim_P I_\aa^\bb(E_\ul) = 1$.

\item[(b)] 
If ${\displaystyle\liminf_{n\to\infty}} \frac{\log\Delta_{n+1}}{\log a_n}<1$, then  $\dim_H I_\aa^\bb(E_\ul) = 1$.

\item[(c)] If ${\displaystyle\liminf_{n\to\infty}} 
\frac{\log \Delta_{n+1}}{\log b_n} \ge 1$, then $\dim_P I_\aa^\bb(E_\ul) = 2$.

\item[(d)] If ${\displaystyle\liminf_{n\to\infty}} 
\frac{\log\Delta_{n+1}}{\log b_n} > 1$, then $\dim_H I_\aa^\bb(E_\ul) = \dim_P I_\aa^\bb(E_\ul) = 2$.
\end{itemize}
\end{cor}

\begin{rem}\label{rem:1-2}
The assertion (b) holds also under the weaker assumption $\liminf_{n\to\infty}\frac{\log\Delta_{n+1}}{\log(a_1 \cdots a_n)} < 1$, while (d) holds also under the weaker assumption $\inf_{n\in\N} \frac{\Delta_{n+1}}{b_n} > 0$.
\end{rem}

\begin{proof}[Proof of Corollary~{\rm \ref{cor:1-2}} and Remark~{\rm \ref{rem:1-2}}]

To prove the assertion (a), note that by assumption, for any $\varepsilon > 0$ there exists $n_0 > 0$ such that $\log\Delta_{n+1} < \varepsilon \log a_n$ for $n \ge n_0$, which gives
\[
\frac{\log\big(\Delta_1\cdots  \Delta_{n+1}\big)}{\log(a_1\cdots a_n)} < \frac{\log\big(\Delta_1\cdots  \Delta_{n_0}\big)}{\log(a_1\cdots a_n)} + \varepsilon\xrightarrow[n\to\infty]{} \varepsilon
\]
since $(a_1 \cdots a_n)^{1/n} \to \infty$. Hence, the assertion (a) follows from Corollary~\ref{cor:dim}. The assertion (b) under the weaker assumption from Remark~\ref{rem:1-2} holds by Theorem~\ref{thm:A}, Remark~\ref{rem:A} and Corollary~\ref{cor:dim}.
To show (c), take a small $\varepsilon > 0$ and note that by assumption, there exists $n_0 > 0$ such that
\[
\log\Delta_{n+1} \ge (1 - \varepsilon) \log b_n
\]
for $n \ge n_0$, so for $x_n \in [a_n, b_n]$, $n \ge n_0$, we have
\[
\log(\min(\Delta_{n+1}, x_n)) \ge \min((1 - \varepsilon) \log b_n, \log x_n) \ge (1 - \varepsilon) \log x_n.
\]
Hence, there exists a constant $C \in \R$, such that for $\psi_n$ from Theorem~\ref{thm:B},
\[
\psi_n(x_1, x_2, \ldots) \ge \frac{C + (1 - \varepsilon) (\log x_{n_0} + \cdots + \log x_n)}{\log x_1 + \cdots + \log x_n} \xrightarrow[n\to\infty]{} 1
\]
since $x_n \ge a_n$ and $(a_1 \cdots a_n)^{1/n} \to \infty$. This implies $\limsup_{n\to\infty} \psi_n \ge 1$, so (c) holds by Theorem~\ref{thm:B}.

To show the assertion (d) under the weaker assumption from Remark~\ref{rem:1-2}, note that if $\inf_{n\in\N} \Delta_{n+1}/b_n > 0$, then there exist $n_0, c > 0$ such that
\[
\log \Delta_{n+1} \ge \log b_n + \log c
\]
for $n \ge n_0$, so for $x_n \in [a_n, b_n]$ we have
\[
\log(\min(\Delta_{n+1}, x_n)) \ge \min( \log b_n + \log c, \log x_n) \ge \log x_n - |{\log c}|.
\]
Hence, for $\phi_n$ from Theorem~\ref{thm:B} and a constant $C \in \R$,
\[
\phi_n(x_1, x_2, \ldots) \ge \frac{C - |{\log c}|n + \log x_{n_0} + \cdots + \log x_{n-1}}{\log x_1 + \cdots + \log x_{n-1} - |{\log c}|} \xrightarrow[n\to\infty]{} 1,
\]
as  $x_n \ge a_n$ and$(a_1 \cdots a_n)^{1/n} \to \infty$. This gives $\limsup_{n\to\infty} \phi_n \ge 1$, and (d) holds by Theorem~\ref{thm:B}. Note that $\inf_{n\in\N} \frac{\Delta_{n+1}}{b_n} > 0$ is indeed a weaker assumption, since the condition $\liminf_{n\to\infty} \frac{\log\Delta_{n+1}}{\log b_n} > 1$ implies
\[
\frac{\Delta_{n+1}}{b_n} > b_n^c \ge a_n^c \ge a^c
\]
for large $n$ and a constant $c > 0$.
\end{proof}

\subsection{Moderately slow escaping points}\label{subsec:moder}

We extend the notion of the moderately slow escaping set to the non-autonomous setting. 

\begin{defn}
Let
\[
M(E_\ul) = \Big\{z \in I(E_\ul) : \limsup_{n\to\infty} \frac 1 n \log\log |E_{\lambda_n} \circ \cdots \circ E_{\lambda_1}(z)| < \infty\Big\} 
\]
be the \emph{moderately slow escaping set} of $E_\lambda$. 
\end{defn}

The following result shows that if the considered set $I_\aa^\bb(E_\ul)$ is contained in the moderately slow escaping set, then its Hausdorff dimension is equal to $1$. In particular, this generalizes the results (a)--(d) from \cite{sixsmith}, mentioned in the introduction, since the sets considered in \cite{sixsmith} are contained in the moderately slow escaping set.

\begin{thm}\label{thm:any_a} Under the assumptions of Theorem~{\rm\ref{thm:B}}, 
if $\inf_{n\in\N}(\log^+ b_n)^{1/n} < \infty$, then 
\[
\dim_H I_\aa^\bb(E_\ul) = 1.
\]
In particular, this holds if $I_\aa^\bb(E_\ul)$ is contained in the moderately slow escaping set $M(E_\ul)$ and the assumptions of Theorem~{\rm\ref{thm:B}} are satisfied.
\end{thm}
\begin{rem} \label{rem:any_a} By Theorem~\ref{thm:A}, the fact $\dim_H I_\aa^\bb(E_\ul) \le 1$ holds under weaker assumptions $\inf_{n\in\N} a_n > 0$, $\lim_{n\to\infty} (a_1\cdots a_n)^{1/n} = \infty$ and $\inf_{n\in\N}(\log^+ b_n)^{1/n} < \infty$.
\end{rem}
\begin{proof}[Proof of Theorem~{\rm\ref{thm:any_a}}] The first assertion follows from Theorem~\ref{thm:A} and Corollary~\ref{cor:dim} in a straightforward way. To show the second one, note that if $z \in I_\aa^\bb(E_\ul) \subset M(E_\ul)$, then $a_n \le |z| \le \min(b_n, e^{e^{cn}})$ for large $n$ and a constant $c > 1$. Hence, if $I_\aa^\bb(E_\ul) \subset M(E_\ul)$, then $I_\aa^\bb(E_\ul)$ is contained in a countable union of the sets of the form $I_\aa^{\bb'}(E_\ul)$, where
\[
\bb' = (b'_n)_{n=1}^\infty \qquad \text{for } b_n' = \min(b_n, e^{e^{cn}}),
\]
for some $c > 1$. Since $(\log^+ b'_n)^{1/n} \le e^c$, Theorem~\ref{thm:A} implies $\dim_H I_\aa^\bb(E_\ul) \le 1$. The opposite inequality follows from Theorem~\ref{thm:B}.
\end{proof}

\subsection{Points with exact growth rate}\label{subsec:rate}

Our results enable to determine the Hausdorff and packing dimension of the set of points, which share the same growth rate under iteration of $E_\ul$.

\begin{defn}
We say that the iterations of a point $z \in \C$ under $E_\ul$ \emph{have growth rate} $\aa$ for a sequence $\aa = (a_n)_{n=1}^\infty$, if $a_n/c \le |E_{\lambda_n} \circ \cdots \circ E_{\lambda_1}(z)| \le ca_n$ for large $n$ and some constant $c > 1$, i.e.~$z \in I_{\aa/c}^{c\aa}(E_\ul)$.
\end{defn}

Corollary~\ref{cor:1-2} immediately implies the following.

\begin{thm}\label{thm:rate} If $a_{n+1} \le |\lambda_{n+1}|e^{q a_n}$ for large $n$ and some constant $0 < q < 1$, $(a_1 \cdots a_n)^{1/n} \to \infty$ as $n \to \infty$ and $\liminf_{n \to \infty} a_n > a$, where $a$ is a sufficiently large constant depending on $q$, then the set of points with growth rate $\aa$ has Hausdorff dimension $1$. If
$a_{n+1} \le |\lambda_{n+1}|e^{q a_n}$ for large $n$ and $a_n \to \infty$ as $n \to \infty$, then the set of points with growth rate $\aa$ has Hausdorff and packing dimension $1$.

\end{thm}

\subsection{Precise dimension formulas}\label{subsec:growth}

In the case $\sup_{n\in\N} |\lambda_n| < \infty$, $\lim_{n\to\infty}\frac{\log b_n}{\log a_n} = 1$ we can exactly determine the Hausdorff and packing dimension of $I_\aa^\bb(E_\ul)$.

\begin{thm}\label{thm:thin} Under the assumptions of Theorem~{\rm\ref{thm:B}}, if ${\displaystyle\sup_{n\in\N}} |\lambda_n| < \infty$ and ${\displaystyle\lim_{n\to\infty}}\frac{\log b_n}{\log a_n} = 1$, then
\[
\dim_H I_\aa^\bb(E_\ul) = 1, \qquad \dim_P I_\aa^\bb(E_\ul) = 1 + \limsup_{n\to\infty}\frac{\log(\Delta_1\cdots \Delta_{n+1})}{\log(a_1\cdots a_n)}.
\]
\end{thm}
\begin{proof} If $\lim_{n\to\infty}\frac{\log b_n}{\log a_n} = 1$, then
\[
\lim_{n\to\infty}\frac{\log (b_1\cdots b_n)}{\log (a_1 \cdots a_n)} = 1
\]
by the Stolz--Ces\`aro Theorem. Therefore, the theorem follows directly from  Corollary~\ref{cor:dim} and Remark~\ref{rem:dim}.
\end{proof}

The following result provides examples of sets $I_\aa^\bb(E_\ul)$ with packing dimension equal to any given value in the interval $[1,2]$. 

\begin{thm} \label{thm:examples}
For every $D \in [1, 2]$ and every sequence $(\lambda_n)_{n=1}^\infty$ with $\lambda_n \in \C \setminus \{0\}$, such that $\sup_{n\in\N} |\lambda_n| < \infty$, there exist admissible sequences $\aa = (a_n)_{n=1}^\infty$, $\bb = (b_n)_{n=1}^\infty$ with $a_n \to \infty$, $\inf_{n\in\N}\Delta_n > 0$ and $\lim_{n\to \infty} \frac{\log b_n}{\log a_n} = 1$, such that 
\[
\dim_H I_\aa^\bb(E_\ul) = 1, \qquad \dim_P I_\aa^\bb(E_\ul)  = D.
\]
\end{thm}

Theorem~\ref{thm:examples} is implied by the following corollary, which is a direct consequence of Theorem~\ref{thm:thin} and 
the Stolz--Ces\`aro Theorem. 

\begin{cor}\label{cor:thin2} Under the assumptions of Theorem~{\rm\ref{thm:B}}, if $\sup_{n\in\N} |\lambda_n| < \infty$, $\lim_{n\to\infty}\frac{\log b_n}{\log a_n} = 1$ and $\lim_{n\to\infty} \frac{\log\Delta_{n+1}}{\log a_n} = d$ for $d \in [0, 1]$, then $\dim_H I_\aa^\bb(E_\ul) = 1$ and $\dim_P I_\aa^\bb(E_\ul) = 1 + d$.
\end{cor}

The following example shows that the assumptions of Corollary~\ref{cor:thin2} are actually satisfied for some sequences $(a_n)_{n=1}^\infty$, $(b_n)_{n=1}^\infty$, which proves Theorem~\ref{thm:examples}.

\begin{ex}\label{ex:dim_P} For any sequence $(\lambda_n)_{n=1}^\infty$ with $\sup_{n\in\N} |\lambda_n| < \infty$, $a_{n+1} = e^{n a_n^d}$ for $d \in [0,1)$ and $b_n = a_n^{1 + \frac{1}{n}}$, then $\dim_H I_\aa^\bb(E_\ul) = 1$, $\dim_P I_\aa^\bb(E_\ul)  = 1 + d$. If $a_{n+1} = e^{n a_n^{(n-1)/n}}$,
$b_n = a_n^{1 + \frac{1}{n}}$, then $\dim_H I_\aa^\bb(E_\ul) = 1$, $\dim_P I_\aa^\bb(E_\ul) = 2$.
\end{ex}
\begin{proof}It is a direct calculation to check that $(a_n)_{n=1}^\infty$, $(b_n)_{n=1}^\infty$ satisfy the assumptions of Corollary~\ref{cor:thin2}.
\end{proof}

We end this section by stating a question, which we find interesting to determine.

\begin{quest} Does there exist a set $I_\aa^\bb(E_\ul)$ with $\dim_H I_\aa^\bb(E_\ul) \in (1,2)$?
\end{quest}

\section{Annular itineraries}\label{sec:annular}

Sets of the form $I_\aa^\bb(f)$ appear naturally in the study of \emph{annular itineraries} $\underline{s}(z) = (s_n)_{n=0}^\infty$ of points $z \in \C$ under a map $f\colon \C \to \C$, defined by
\[
f^n(z) \in \AAA_{s_n}, \quad n \ge 0, \qquad \text{where} \quad
\AAA_s = \{z \in \C: R_s \le |z| < R_{s+1}\}, \quad s \ge 0,
\]
for some sequence $0 = R_0 < R_1 < R_2 < \cdots$, with $R_s \to \infty$ as $s \to \infty$. Such annular itineraries, for $R_s = M_f^{s-1}(R_1)$, were studied by Rippon and Stallard in \cite{RS-annular}. In \cite{sixsmith}, Sixsmith, considering exponential maps, used the annuli defined by $R_s = R^s$ for a large $R > 1$. He proved that if $s_n \to \infty$, then the set of points sharing the itinerary $\underline{s}(z) = (s_n)_{n=0}^\infty$ has Hausdorff dimension at most $1$, while the dimension is equal to $1$, if additionally, $R$ is sufficiently large, $\underline{s}$ is slowly-growing, i.e.~$\frac{s_{n+1}}{s_1 + \cdots + s_n} \to 0$ and $\underline{s}$ is admissible, in the sense that $s_{n+1} < e^{s_n}$. Here we extend the results, answering a question from \cite{sixsmith} and showing that the assumption of the slow growth can be omitted. We also analyse annular itineraries defined by another partition of the plane, given by $R_s = R^{s^\kappa}$ for $\kappa > 1$. In this case one can find examples of the sets of points sharing the same itinerary, with the packing dimension larger than $1$.

We extend the notion of annular itineraries to the non-autonomous setup, setting
\[
\underline{s}(z) = (s_n)_{n=0}^\infty, \quad \text{where} \quad E_{\lambda_n} \circ \cdots \circ E_{\lambda_1}(z) \in \AAA_{s_n}.
\]
We assume $s_n > 0$ for $n \ge 0$. For given symbolic sequence $\underline{s} = (s_n)_{n=0}^\infty$ let
\[
\II_{\underline{s}}(E_\ul) = \{z \in \C : \underline{s}(z) = \underline{s}\}
\]
Note that
\[
\II_{\underline{s}}(E_\ul) =  I_\aa^\bb(E_\ul) \qquad \text{for} \quad a_n = R_{s_n}, \quad b_n = R_{s_n + 1}.
\]
We will say that a sequence $\underline{s} = (s_n)_{n=0}^\infty$ is \emph{admissible}, if the sequences $\aa= (a_n)_{n=1}^\infty$, $\bb= (a_n)_{n=1}^\infty$ for $a_n = R_{s_n}$, $b_n = R_{s_n + 1}$ are admissible.

\subsection{Case \boldmath $R_s = R^s$}

Consider annular itineraries $\underline{s} = (s_n)_{n=0}^\infty$ of points under non-autonomous iteration $E_\ul$ with respect to the annuli 
\[
\AAA_s = \{z \in \C : R^s \le |z| < R^{s+1}\},
\]
for $s \ge 0$ and $R > 1$. 

\begin{thm}\label{thm:annular} 
The following statements hold.
\begin{itemize}
\item[(a)] If ${\displaystyle\limsup_{n\to\infty}} \frac{s_1+ \cdots + s_n}{n} = \infty$, then $\dim_H \II_{\underline{s}}(E_\ul) \le 1$. 
\item[(b)]
If $\underline{s}$ is admissible, ${\displaystyle\lim_{n\to\infty}} \frac{s_1+ \cdots + s_n}{n} = \infty$ and $R$ is sufficiently large, then 
\[
\dim_H \II_{\underline{s}}(E_\ul) = \dim_P \II_{\underline{s}}(E_\ul) =1.
\]
\end{itemize}
\end{thm}
\begin{proof}
As noted above, we have $\dim \II_{\underline{s}}(E_\ul) = \dim I_\aa^\bb(E_\ul)$ for 
\[
a_n = R^{s_n}, \qquad b_n = R^{s_n + 1}.
\]
In particular, $a_n \ge R$, $a_1 \cdots a_n = R^{s_1+ \cdots + s_n}$ and $\Delta_n =\log R$. Hence, the assertions follow immediately from Theorem~\ref{thm:A} and Corollary~\ref{cor:dim}. 
\end{proof}

Note that the admissibility condition according to Definition~\ref{defn:admis} has the form
\[
s_{n+1} \le \frac{q R^{s_n} + \log|\lambda_{n+1}|}{\log R}
\]
for $0 < q < 1$, and in the non-autonomous case is satisfied provided $s_{n+1} \leq R^{qs_n}$, if $R$ is sufficiently large.
\subsection{Case \boldmath $R_s = R^{s^\kappa}$}

Consider now annular itineraries $\underline{s} = (s_n)_{n=0}^\infty$ with respect to the annuli 
\[
\AAA_s = \{z \in \C : R^{s^\kappa} \le |z| < R^{(s+1)^\kappa}\},
\]
for $s \ge 0$ and $R > 1$, $\kappa > 1$. 

\begin{thm}\label{thm:annular2} Suppose $\sup_{n\in\N} |\lambda_n| < \infty$, $\underline{s}$ is admissible and $R$ is sufficiently large. Then the following statements hold.
\begin{itemize} 
\item[(a)] 
If ${\displaystyle\limsup_{n\to\infty}}\frac{s_1^\kappa + \cdots + s_n^\kappa}{n} = \infty$, then $\dim_H \II_{\underline{s}}(E_\ul) \le 1$.

\item[(b)] 
If ${\displaystyle\lim_{n\to\infty}}s_n = \infty$, then:
\begin{align*}
\dim_H \II_{\underline{s}}(E_\ul) &= 1,\\
\dim_P \II_{\underline{s}}(E_\ul) &= 1 + \frac{\kappa-1}{\log R}\limsup_{n\to\infty}
\frac{\log s_{n+1}}{s_1^\kappa + \cdots + s_n^\kappa},\\
\dim_P \II_{\underline{s}}(E_\ul) &< 2 - \frac 1 \kappa.
\end{align*}

\end{itemize}
\end{thm}
\begin{proof}
In this case we have $\dim \II_{\underline{s}}(E_\ul) = \dim I_\aa^\bb(E_\ul)$ for 
\[
a_n = R^{s_n^\kappa}, \qquad b_n = R^{(s_n + 1)^\kappa}.
\]
In particular, $a_n \ge R^\kappa$ and $a_1 \cdots a_n = R^{s_1^\kappa+ \cdots + s_n^\kappa}$. By the assumption $\sup_{n\in\N} |\lambda_n| < \infty$, the admissibility condition is equivalent to
\begin{equation}\label{eq:ann_adm}
s_{n+1} \le \Big(\frac{q}{\log R}\Big)^\frac{1}{\kappa}R^\frac{s_n^\kappa}{\kappa}
\end{equation}
for large $n$ and a constant $0 < q < 1$. Moreover,
\[
\frac{\log b_n}{\log a_n} = 
\bigg(1 + \frac{1}{s_n}\bigg)^\kappa.
\]
and
\[
\Delta_n  =  ((s_n+1)^\kappa - s_n^\kappa) \log R \ge (\kappa s_n^{\kappa -1})\log R \ge  \kappa \log R.
\]
By the Mean Value Theorem,
\begin{equation}\label{eq:<s<}
(\kappa - 1) \log s_{n+1} - c_1 \le \log\Delta_{n+1} \le (\kappa - 1) \log s_{n+1} + c_1
\end{equation}
for a constant $c_1 > 0$ and by \eqref{eq:ann_adm},
\begin{equation}\label{eq:ann_adm2}
\log s_{n+1} \le \frac{s_n^\kappa}{\kappa} \log R + c_2
\end{equation}
for a constant $c_2 > 0$. Furthermore, \eqref{eq:<s<} and \eqref{eq:ann_adm2} imply
\[
\limsup_{n\to\infty} \frac{\log\Delta_{n+1}}{\log(a_1 \cdots a_n)} \le \frac{\kappa - 1}{\kappa}\limsup_{n\to\infty} \frac{s_n^\kappa + \frac{\kappa}{\kappa - 1}\frac{c_1 + c_2(\kappa -1)}{\log R}}{s_1^\kappa + \cdots + s_n^\kappa}  \le\frac{\kappa - 1}{\kappa} < 1,
\]
which proves (a) by Theorem~\ref{thm:A}, since $\limsup_{n\to\infty} (a_1 \cdots a_n)^{1/n} = \infty$ by assumptions.

The first assertion of (b) follows from (a) and Corollary~\ref{cor:dim}.
To prove the other ones, note that if $s_n \to \infty$, then 
$\frac{\log b_n}{\log a_n} \to 1$, so by Theorem~\ref{thm:thin} and \eqref{eq:<s<},
\[
\dim_P \II_{\underline{s}}(E_\ul) = 1 + \frac{\kappa-1}{\log R}\limsup_{n\to\infty}
\frac{\log s_1 + \cdots + \log s_{n+1}}{s_1^\kappa + \cdots + s_n^\kappa}.
\]
Since
\[
\frac{\log s_1 + \cdots + \log s_n}{s_1^\kappa + \cdots + s_n^\kappa} \to 0,
\]
we have
\[
\dim_P \II_{\underline{s}}(E_\ul) = 1 + \frac{\kappa-1}{\log R}\limsup_{n\to\infty}
\frac{\log s_{n+1}}{s_1^\kappa + \cdots + s_n^\kappa},
\]
and by \eqref{eq:ann_adm2}, $\dim_P \II_{\underline{s}}(E_\ul) < 2 - \frac 1 \kappa$, which proves the second and third assertion of (b). 
\end{proof}

Finally, we provide examples of sets $\II_{\underline{s}}(E_\ul)$ with packing dimension larger than $1$.

\begin{cor}\label{cor:ann_dim_P-ex} Suppose $\sup_{n\in\N} |\lambda_n| < \infty$, $\lim_{n\to\infty}\frac{\log s_{n+1}}{s_n^\kappa} = \frac{d\log R}{\kappa - 1}$ for $d \in [0, 1 - \frac 1 \kappa)$ and $R$ is sufficiently large. Then $\underline{s}$ is admissible and $\dim_H \II_{\underline{s}}(E_\ul) = 1$, $\dim_P \II_{\underline{s}}(E_\ul) = 1 + d$. 
\end{cor}
\begin{proof}
Follows directly from the assertion (b) of Theorem~\ref{thm:annular2} and \eqref{eq:ann_adm}.
\end{proof}

The conditions of Corollary~\ref{cor:ann_dim_P-ex} are actually satisfied for some sequences $(s_n)_{n=0}^\infty$, as shown in the following example.

\begin{ex}\label{ex:ann_dim_P} If $s_{n+1} = R^{\frac d {\kappa - 1}s_n^\kappa}$ for $d \in [0, 1 - \frac 1 \kappa)$, then $\dim_H \II_{\underline{s}}(E_\ul) = 1$, $\dim_P \II_{\underline{s}}(E_\ul) = 1 + d$.
\end{ex}

\section{Proofs of Theorems~\ref{thm:A} and~\ref{thm:B} -- preliminaries}\label{sec:prelim}

We will use the notation
\[
\diam X = \sup\{|x-y|: x, y \in X\}
\]
and
\[
\dist(z, X) = \inf\{|z - x|: x \in X\}, \qquad \dist(X, Y) = \inf\{|x - y|: x \in X, y \in Y\}
\]
for $z \in \C$, $X, Y \subset \C$.

Let
\[
J_N = \{z \in \C: a_{N+n} \le |E_{\lambda_{N+n}} \circ \cdots \circ E_{\lambda_N}(z)| \le b_{N+n} \text{ for every } n \ge 0\}
\]
for $N \in \N$. By definition, 
\[
I_\aa^\bb(E_\ul) = J_1 \cup \bigcup_{N =2}^\infty (E_{\lambda_{N-1}} \circ \cdots \circ E_{\lambda_1})^{-1}(J_N)
\]
and
\[
J_{N_1} \subset (E_{\lambda_{N_2-1}} \circ \cdots \circ E_{\lambda_{N_1}})^{-1}(J_{N_2})
\]
for every $1 \le N_1 < N_2$. As $E_{\lambda_n}$ are non-constant holomorphic maps, we have
$\dim J_{N_1} \le \dim J_{N_2}$ for $N_1 < N_2$ and 
\begin{equation}\label{eq:dim}
\dim I_\aa^\bb(E_\ul) = \dim \Big(\bigcup_{N = 1}^\infty J_N\Big) = \sup_{N \in \N}\dim  J_N = \lim_{N \to \infty}\dim J_N,
\end{equation}
where $\dim$ denotes the Hausdorff or packing dimension. Therefore, to estimate the dimensions of the sets $I_\aa^\bb(E_\ul)$ it is sufficient to bound the suitable dimensions of $J_N$ for large $N$. 

From now on, we fix a large $N$ and write $J$ for $J_N$. For $n \ge 0$ let
\[
A_n = \log \frac{a_{N+n}}{|\lambda_{N+n}|}, \qquad B_n = \log \frac{b_{N+n}}{|\lambda_{N+n}|}
\]
and
\[
S_n = \{z\in\C: A_n \le \Re(z) \le B_n\}
\]
for $n \in \N$. Recall that
\[
B_n - A_n = \Delta_{N+n} = \log\frac{b_{N+n}}{a_{N+n}}.
\]
Note that 
\[
z \in S_n \iff a_{N+n} \le |E_{\lambda_{N+n}}(z)| \le b_{N+n},
\]
so
\begin{equation}
\label{eq:J}
J = \{z \in \C: z \in S_0, \; E_{\lambda_{N+n}} \circ \cdots \circ E_{\lambda_N}(z) \in S_{n+1} \text{ for every } n \ge 0\}.
\end{equation}
For a small $\delta > 0$ and $j, k \in \Z$ let
\begin{align*}
V_j^{(n)} &= \{z\in\C: j\delta - \log|\lambda_{N+n}| \le \Re(z) < (j+1)\delta - \log|\lambda_{N+n}|\},\\
H^{(n)}_k &= \{z\in\C: k\pi - \Arg (\lambda_{N+n}) \le \Im(z) < (k+1)\pi - \Arg (\lambda_{N+n})\}
\end{align*}
with $\Arg (\lambda_{N+n}) \in [0, 2\pi)$. Set
\[
K^{(n)}_{j,k} = V_j^{(n)} \cap H^{(n)}_k.
\]
Note that 
\begin{equation}\label{eq:<j<}
\text{if} \quad K^{(n)}_{j,k} \cap S_n \ne \emptyset, \quad \text{then} \quad
\frac{\log a_{N + n}}{\delta} - 1 < j < \frac{\log b_{N + n}}{\delta},  \quad \text{so} \quad e^{-\delta} a_{N + n} < e^{j\delta} < b_{N + n}. 
\end{equation}
We have
\[
E_{\lambda_{N+n}}(K^{(n)}_{j,k}) = U_{j,k}
\]
for
\[
U_{j,k} = \{z \in \C: e^{j\delta} \le |z| < e^{(j+1)\delta}, \; k\pi \le \Arg(z) < (k+1)\pi \text{ mod } 2\pi\}. 
\]
Note that
\[
U_{j,k + 2} = U_{j,k}. 
\]
Set
\[
\KK^{(n)} = \{K^{(n)}_{j,k} : j, k \in \Z\}.
\]
for $n \ge 0$ and
\begin{align*}
\KK_{j,k}^{(n)} &= \{K \in \KK^{(n)}: K \cap U_{j,k} \cap S_n \neq \emptyset\},\\
\tilde\KK_{j,k}^{(n)} &= \{K \in \KK^{(n)}: K \subset U_{j,k} \cap S_n\}
\end{align*}
for $n \ge 1$. Obviously,
\[
\tilde\KK_{j,k}^{(n)} \subset \KK_{j,k}^{(n)}.
\]
Let 
\[
Q_k^{(n)} = \{z \in \C: z \in S_n, \; \Delta_{N+n} k \le \Im(z) \le \Delta_{N+n}(k+1)\}
\]
for $n \ge 0$, $k \in \Z$ and
\[
\QQ_{j,k}^{(n)} = \{Q_l^{(n)} : Q_l^{(n)} \cap U_{j, k} \ne \emptyset,\: l \in \Z\}.
\]
for $n \ge 0$, $k, j \in \Z$.
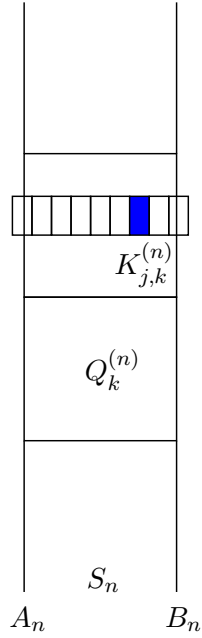
\begin{figure}[!ht]
\begin{center}
\def\svgwidth{0.15\textwidth}
\begingroup%
  \makeatletter%
  \providecommand\color[2][]{%
    \errmessage{(Inkscape) Color is used for the text in Inkscape, but the package 'color.sty' is not loaded}%
    \renewcommand\color[2][]{}%
  }%
  \providecommand\transparent[1]{%
    \errmessage{(Inkscape) Transparency is used (non-zero) for the text in Inkscape, but the package 'transparent.sty' is not loaded}%
    \renewcommand\transparent[1]{}%
  }%
  \providecommand\rotatebox[2]{#2}%
  \newcommand*\fsize{\dimexpr\f@size pt\relax}%
  \newcommand*\lineheight[1]{\fontsize{\fsize}{#1\fsize}\selectfont}%
  \ifx\svgwidth\undefined%
    \setlength{\unitlength}{34.03345376bp}%
    \ifx\svgscale\undefined%
      \relax%
    \else%
      \setlength{\unitlength}{\unitlength * \real{\svgscale}}%
    \fi%
  \else%
    \setlength{\unitlength}{\svgwidth}%
  \fi%
  \global\let\svgwidth\undefined%
  \global\let\svgscale\undefined%
  \makeatother%
  \begin{picture}(1,3.61803777)%
    \lineheight{1}%
    \setlength\tabcolsep{0pt}%
    \put(0.42437878,0.29076926){\color[rgb]{0,0,0}\makebox(0,0)[lt]{\lineheight{1.25}\smash{\begin{tabular}[t]{l}$S_n$\end{tabular}}}}%
    \put(0.57900278,2.07748598){\color[rgb]{0,0,0}\makebox(0,0)[lt]{\lineheight{1.25}\smash{\begin{tabular}[t]{l}$K_{j,k}^{(n)}$\end{tabular}}}}%
    \put(0.41224201,1.47009062){\color[rgb]{0,0,0}\makebox(0,0)[lt]{\lineheight{1.25}\smash{\begin{tabular}[t]{l}$Q_k^{(n)}$\end{tabular}}}}%
    \put(-0.0121806,0.07004962){\color[rgb]{0,0,0}\makebox(0,0)[lt]{\lineheight{1.25}\smash{\begin{tabular}[t]{l}$A_n$\end{tabular}}}}%
    \put(0.86491356,0.06929629){\color[rgb]{0,0,0}\makebox(0,0)[lt]{\lineheight{1.25}\smash{\begin{tabular}[t]{l}$B_n$\end{tabular}}}}%
    \put(0,0){\includegraphics[width=\unitlength,page=1]{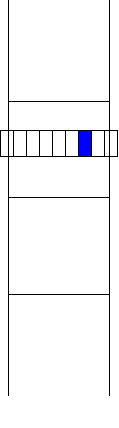}}%
  \end{picture}%
\endgroup%

\caption{The sets $S_n$, $Q_k^{(n)}$ and $K_{j,k}^{(n)}$.}\label{fig:s}
\end{center}
\end{figure}
Finally, let
\[
U_k = \bigcup_{j \in \Z} U_{j,k} =  \{z \in \C \setminus \{0\}: k\pi \le \Arg(z) < (k+1)\pi\}
\]
and 
\begin{equation}\label{eq:g_k}
g^{(n)}_k \colon  U_k \to H^{(n)}_k
\end{equation}
for $k \in \Z$ be inverse branches of $E_{\lambda_{N+n}}$ on $U_k$. Note that $g^{(n)}_k$ can be extended to any simply connected domain in $\C \setminus \{0\}$ containing $U_k$.

\begin{figure}[!ht]
\begin{center}
\def\svgwidth{0.98\textwidth}
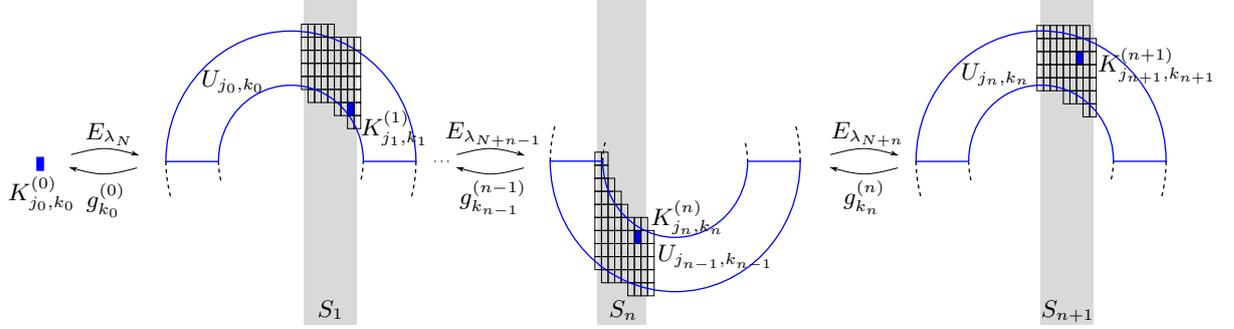
\caption{Successive images of the sets $K_{j,k}^{(n)}$.}\label{fig:k}
\end{center}
\end{figure}

\section{Proof of Theorem~\ref{thm:A}}\label{sec:proof-A}

Fix $j_0, k_0 \in \Z$ and take $j_1, \ldots, j_n \in \Z$, $k_1, \ldots, k_n \in \Z$ and $l \in \Z$ such that 
\[
K^{(1)}_{j_1,k_1} \in \KK^{(1)}_{j_0,k_0}, \ldots, K^{(n)}_{j_n,k_n} \in \KK^{(n)}_{j_{n-1}, k_{n-1}}, Q^{(n+1)}_l  \in \QQ^{(n+1)}_{j_n, k_n}.
\]
Define inductively
\begin{align*}
Q_{j_n, k_n, l} &= g^{(n)}_{k_n}(Q^{(n+1)}_l \cap U_{j_n, k_n}), &&\\
Q_{j_{m-1}, k_{m-1}, \ldots, j_n,k_n,l} &= g^{(m-1)}_{k_{m-1}}(Q_{j_m, k_m, \ldots, j_n,k_n,l} \cap U_{j_{m-1}, k_{m-1}}) &&\text{for } m = n, \ldots, 1.
\end{align*}
and let
\begin{align*}
\EE^{(n)} = \{&Q_{j_0, k_0, \ldots, j_n,k_n,l} : 
K^{(1)}_{j_1,k_1} \in \KK^{(1)}_{j_0,k_0}, \ldots, K^{(n)}_{j_n,k_n} \in \KK^{(n)}_{j_{n-1}, k_{n-1}}, Q^{(n+1)}_l  \in \QQ^{(n+1)}_{j_n, k_n},\\
&j_1, \ldots, j_n \in \Z, \; k_1, \ldots, k_n \in \Z, \; l \in \Z \}.
\end{align*}
for $n \in \N$.

\begin{lem}\label{lem:cover}
For every $n \in\N$ the family $\EE^{(n)}$ is a cover of $J \cap K^{(0)}_{j_0, k_0}$. 
\end{lem}

\begin{proof}
Take $z \in J \cap K^{(0)}_{j_0, k_0}$. By~\eqref{eq:J}, for every $m \ge 1$, $E_{\lambda_{N+m-1}} \circ \cdots \circ E_{\lambda_N}(z) \in K^{(m)}_{j_m,k_m} \cap S_m$ for some $j_m, k_m \in \Z$. Hence, for given $n \in \N$,
$E_{\lambda_{N+m-1}} \circ \cdots \circ E_{\lambda_N}(z) \in  K^{(m)}_{j_m,k_m} \cap U_{j_{m-1}, k_{m-1}} \cap S_m$ for $m = 1, \ldots, n$ and $E_{\lambda_{N+n}} \circ \cdots \circ E_{\lambda_N}(z) \in Q^{(n+1)}_l \cap U_{j_n, k_n}$ for some $l \in \Z$. Therefore, 
$K^{(m)}_{j_m,k_m} \in \KK^{(m)}_{j_{m-1},k_{m-1}}$ for $m = 1, \ldots, n$ and $Q^{(n+1)}_l  \in \QQ^{(n+1)}_{j_n, k_n}$. By induction, 
\[
g^{(m)}_{k_m}(E_{\lambda_{N+m}} \circ \cdots \circ E_{\lambda_N}(z)) = E_{\lambda_{N+m-1}} \circ \cdots \circ E_{\lambda_N}(z) \in Q_{j_m, k_m, \ldots, j_n,k_n,l} \cap U_{j_{m-1}, k_{m-1}}
\]
for $m = n, \ldots, 1$ and 
\[
g^{(0)}_{k_0}(E_{\lambda_N}(z)) = z \in Q_{j_0, k_0, \ldots, j_n,k_n,l}\in \EE^{(n)}.
\]
\end{proof}

By \eqref{eq:<j<} we can write
\begin{equation}\label{eq:EE_j}
\EE^{(n)} \subset \bigcup_{\frac{\log a_{N+1}}{\delta}-1  < j_1 < \frac{\log b_{N+1}}{\delta}} \cdots\bigcup_{\frac{\log a_{N+n}}{\delta} -1 < j_n < \frac{\log b_{N+n}}{\delta}} 
\EE^{(n)}_{j_1, \ldots, j_n},
\end{equation}
where
\[
\EE^{(n)}_{j_1, \ldots, j_n} = \{Q_{j_0, k_0, \ldots, j_n,k_n,l} \in \EE^{(n)} : 
k_1, \ldots, k_n \in \Z, \; l \in \Z \}.
\]
Take $j_1, \ldots, j_n$ as in \eqref{eq:EE_j}. If $Q_{j_0, k_0, \ldots, j_n,k_n,l} \in \EE^{(n)}_{j_1, \ldots, j_n}$, then $K^{(m)}_{j_m,k_m} \cap U^{(m-1)}_{j_{m-1}, k_{m-1}} \ne \emptyset$ for $m  = 1, \ldots, n$, so 
\[
K^{(m)}_{j_m,k_m} \subset \{z \in \C: |\Im(z)| < e^{(j_{m-1}+1)\delta}\}.
\]
Moreover, by \eqref{eq:<j<} and the assumption $\liminf a_n > 0$, we can assume $e^{j_{m-1}\delta} > a$ for a constant $a > 0$.
This implies
\[
\#\{k_m \in \Z : Q_{j_0, k_0, \ldots, j_n,k_n,l} \in \EE^{(n)}_{j_1, \ldots, j_n}\} \le c_1 e^{j_{m-1}\delta}
\]
for every $k_0, \ldots, k_{m-1}, k_{m+1}, \ldots, k_n, l \in \Z$ and some constant $c_1 > 0$, so
\[
\#\{(k_1, \ldots, k_n) \in \Z^n : Q_{j_0, k_0, \ldots, j_n,k_n,l} \in \EE^{(n)}_{j_1, \ldots, j_n}\} \le c_1^n e^{(j_0+ \cdots + j_{n-1})\delta}
\]
for every $l \in \Z$. Similarly, $Q^{(n+1)}_l \in \QQ^{(n+1)}_{j_n, k_n}$, so $Q^{(n+1)}_l \cap U^{(n)}_{j_n, k_n} \ne \emptyset$ and
\[
Q^{(n+1)}_l \subset \{z \in \C: |\Im(z)| < e^{j_n\delta} + \Delta_{N+n+1}\},
\]
which gives
\[
\#\{l \in \Z : Q_{j_0, k_0, \ldots, j_n,k_n,l} \in \EE^{(n)}_{j_1, \ldots, j_n}\} < \frac{e^{j_n\delta}}{\Delta_{N+n+1}} + 2
\]
for every $k_0, \ldots, k_n \in \Z$. We conclude that
\begin{equation}\label{eq:card}
\#\EE^{(n)}_{j_1, \ldots, j_n} \le c_1^n e^{(j_0 + \cdots +j_n)\delta}\Big(\frac{1}{\Delta_{N+n+1}} + \frac{2}{e^{j_n\delta}}\Big).
\end{equation}

Now we estimate the diameter of the sets $Q_{j_0, k_0, \ldots, j_n,k_n,l} \in \EE^{(n)}_{j_1, \ldots, j_n}$. We have
\[
\diam Q_{j_n, k_n, l} \le \sup_{U_{j_n, k_n}} |(g^{(n)}_{k_n})'| \diam (Q^{(n+1)}_l \cap U_{j_n, k_n}) \le \frac{1}{e^{j_n\delta}} \min(\sqrt{2}\Delta_{N+n+1}, e^{(j_n+1)\delta}). 
\]
Note also that any two points $z_1, z_2$ in $U_{j, k}$ for $j, k \in \Z$, can be joined within $U_{j, k}$ by a circle arc of length at most $2\pi|z_1-z_2|$. Hence,
\begin{align*}
\diam Q_{j_{m-1}, k_{m-1}, \ldots, j_n,k_n,l} &\le 2\pi  \sup_{U_{j_{m-1}, k_{m-1}}} |(g^{(m-1)}_{k_{m-1}})'| \diam Q_{j_m, k_m, \ldots, j_n,k_n,l}\\ &= \frac{2\pi}{e^{j_{m-1}\delta}} \diam Q_{j_m, k_m, \ldots, j_n,k_n,l}
\end{align*}
for $m = 1, \ldots, n$, which implies
\begin{equation} \label{eq:diam2}
\diam Q_{j_0, k_0, \ldots, j_n,k_n,l} \le (2\pi)^n \frac{\min(\sqrt{2}\Delta_{N+n+1}, e^{(j_n+1)\delta})}{e^{(j_0 + \cdots + j_n)\delta}}.
\end{equation}

Fix $D > 1$ and let  
\[
P^{(n)}_{j_1, \ldots, j_n} = \sum_{Q \in \EE^{(n)}_{j_1, \ldots, j_n}} (\diam Q)^D, \qquad P^{(n)} = \sum_{Q \in \EE^{(n)}} (\diam Q)^D
\]
for $n \in \N$.
By \eqref{eq:card} and \eqref{eq:diam2},
\[
P^{(n)}_{j_1, \ldots, j_n} \le c_2^n
\frac{\Big(\frac{1}{\Delta_{N+n+1}} + \frac{2}{e^{j_n\delta}}\Big)\left(\min(\sqrt{2}\Delta_{N+n+1}, e^{(j_n+1)\delta}\right)^D}{e^{(j_0 + \cdots + j_n))\delta(D-1)}} \le c_3^n\bigg(\frac{\Delta_{N+n+1}}{e^{(j_0 + \cdots + j_n)\delta}}\bigg)^{D-1}
\]
for some constants $c_2, c_3 > 0$ (the latter estimate is by a straightforward calculation).
Hence, by \eqref{eq:EE_j},
\begin{align*}
P^{(n)} &\le c_3^n \sum_{\frac{\log a_{N+1}}{\delta}-1  < j_1 < \frac{\log b_{N+1}}{\delta}} 
\cdots \sum_{\frac{\log a_{N+n}}{\delta} -1 < j_n < \frac{\log b_{N+n}}{\delta}}
\bigg(\frac{\Delta_{N+n+1}}{e^{(j_0 + \cdots + j_n)\delta}}\bigg)^{D-1}\\
&\le c_4^n \bigg(\frac{\Delta_{N+n+1}}{a_N \cdots a_{N+n}}\bigg)^{D-1}
\end{align*}
for some constant $c_4 > 0$. By assumption, for infinitely many $n$ we have
\[
\frac{\Delta_{N+n+1}}{a_N \cdots a_{N+n}} < \varepsilon_n^n,
\]
where $\varepsilon_n > 0$, $\varepsilon_n \to 0$,
so
\[
P^{(n)} \le (c_4\varepsilon_n^{D-1})^n < \frac{1}{2^n}
\]
for infinitely many $n$, so $\liminf_{n\to\infty} P^{(n)} = 0$. Recall that by Lemma~\ref{lem:cover}, $\EE^{(n)}$ is a sequence of covers of $J \cap K^{(0)}_{j_0, k_0}$. Hence, by the definition of the Hausdorff measure we have $\dim_H(J \cap K^{(0)}_{j_0, k_0}) \leq D$ for any $j_0, k_0 \in \Z$ and $D > 1$, so in fact 
$\dim_H J \leq 1$. By \eqref{eq:dim}, $\dim_H I_\aa^\bb(E_\ul) \le 1$, which proofs Theorem~\ref{thm:A}. 

\section{Proof of Theorem~\ref{thm:B} -- preliminaries}\label{sec:proof-B-prelim}

Observe first that if $N$ is chosen large enough, then the assumptions of Theorem~\ref{thm:B} can be written as
\begin{align}
\label{eq:ass-adm}&|\lambda_{N+n+1}| e^{-q a_{N+n}} \le a_{N+ n+1}\le |\lambda_{N+n+1}|e^{q a_{N+n}},\\
\label{eq:ass-infty1}&\lim_{n\to\infty} (a_N \cdots a_{N+n})^{1/n} = \infty,\\\label{eq:ass-a>} &a_{N+n} > a,\\ 
\label{eq:ass-Delta>}&\Delta_{N+n} > \Delta
\end{align}
for $n \ge 0$ and some constants $0 < q < 1$, $a > 0$, $\Delta > 0$, where $a$ is sufficiently large depending on $q$ and $\Delta$ (to be specified later). We fix $\delta$ used in the definition of the sets $K^{(n)}_{j,k}$ to be a positive number such that
\begin{equation}\label{eq:delta<}
\delta < \min(\Delta/4, 1), \qquad \sqrt{q} e^\delta < 1.
\end{equation}

For $n \ge 0$, $j \in \Z$ let
\[
D^{(n)}_{j} = \min (B_n, e^{(j+1)\delta}) - \max (A_n, -e^{(j+1)\delta}).
\]

The following lemma estimates the size of sets $U_{j,k} \cap S_n$. 

\begin{lem}\label{lem:D} There exist $c_1, c_2, c_3 > 0$ such that for every $n\ge 1$, $j, k \in \Z$, 
\begin{itemize}
\item[(a)]
$U_{j,k} \cap S_n$ is contained in a rectangle with horizontal and vertical sides, of width $D^{(n)}_{j}$ and height $e^{(j+1)\delta}$.
\item[(b)] If $K^{(n-1)}_{j, k} \cap S_{n-1} \ne \emptyset$, then 
$U_{j,k} \cap S_n$ contains a rectangle with horizontal and vertical sides, of width $c_1 D^{(n)}_{j}$ and height $c_1 e^{j\delta}$. Moreover, $\KK^{(n)}_{j, k}$ is non-empty and contains a set $K^{(n)}_{j', k'}$ with $j', k' \in 2\Z$.
\item[(c)] If $K^{(n-1)}_{j, k} \cap S_{n-1} \ne \emptyset$, then 
\[
c_3 \le \frac{1}{c_2} \min(\Delta_{N+n}, e^{j\delta}) \le  D^{(n)}_j \le c_2 \min(\Delta_{N+n}, e^{j\delta})  \le c_2e^{j\delta}.
\]
\end{itemize}
\end{lem}
\begin{proof} Let
\[
\tilde A_n = \max (A_n, -e^{(j+1)\delta}), \qquad \tilde B_n = \min (B_n, e^{(j+1)\delta})
\]
for $n \ge 0$, $j \in \Z$.
By the definition of $D^{(n)}_{j}$, 
\begin{equation}\label{eq:tildeAB}
D^{(n)}_{j} = \tilde B_n - \tilde A_n, \qquad \{z \in \C: \Re(z) \in [\tilde A_n, \tilde B_n]\} \subset S_n
\end{equation}
and
\[
U_{j,k} \cap S_n \subset 
\begin{cases}\{z \in \C: \Re(z) \in [\tilde A_n, \tilde B_n], \:\Im(z) \in [0, e^{(j+1)\delta}]\}&\text{if $k$ is even} \\
\{z \in \C: \Re(z) \in [\tilde A_n, \tilde B_n], \:\Im(z) \in [-e^{(j+1)\delta}, 0]\}&\text{if $k$ is odd}
\end{cases},
\]
which together with \eqref{eq:<j<}, \eqref{eq:ass-a>} and \eqref{eq:ass-Delta>} gives the assertion (a). Note also that by \eqref{eq:ass-adm} and \eqref{eq:delta<}, we have 
\begin{equation}\label{eq:<A<}
\tilde A_n < q e^{(j+1)\delta} < \sqrt{q} e^{j\delta} < e^{j\delta}, \qquad 
\tilde B_n > -q e^{(j+1)\delta} > -\sqrt{q} e^{j\delta} > - e^{j\delta}.
\end{equation}
This together with \eqref{eq:<j<}, \eqref{eq:ass-a>} and \eqref{eq:ass-Delta>} gives the assertion (c). Moreover, \eqref{eq:<A<} implies that the vertical line $\{z \in \C : \Re(z) = (\tilde A_n + \tilde B_n)/2\}$ intersects the circle $\bd\D(0, ((e^\delta + 1)/2)e^{j\delta})$ at some point $z_0$. Then the upper (resp.~lower) half of the disc $\D(z_0, ((e^\delta - 1)/2)e^{j\delta})$ is contained in $U_{j,k}$ for even (resp.~odd) $k$. It follows that $U_{j,k} \cap S_n$ contains a rectangle of width $\min( D^{(n)}_{j}, \sqrt{2}\frac{e^{\delta}-1}{2}e^{j\delta})$ and height $\frac{\sqrt{2}}{2}\frac{e^{\delta}-1}{2}e^{j\delta}$. This together with the assertion (c) proves the first part of (b). To show the second part of (b), it is enough to notice that by \eqref{eq:<j<}, \eqref{eq:ass-a>}, \eqref{eq:ass-Delta>}, \eqref{eq:delta<}, \eqref{eq:<A<} and the definition of $D^{(n)}_j$,
\[
\min\Big( D^{(n)}_{j}, \sqrt{2}\frac{e^{\delta}-1}{2}e^{j\delta}\Big) \ge \min\Big(\Delta_{N+n}, \sqrt{2}\frac{e^{\delta}-1}{2}e^{j\delta}\Big) > 4\delta, \qquad \frac{\sqrt{2}}{2}\frac{e^{\delta}-1}{2}e^{j\delta} > 4 \pi,
\]
if $a$ is chosen sufficiently large.
\end{proof}

We will also need the following technical lemma. 

\begin{lem}\label{lem:triangle} Suppose $K^{(n-1)}_{j, k} \cap S_{n-1} \ne \emptyset$ for some $n \ge 1$, $j, k \in \Z$ and
\begin{equation}\label{eq:B}
|A_n + e^{j\delta}| > \varepsilon  e^{j\delta}, \qquad |B_n - e^{j\delta}| > \varepsilon e^{j\delta} 
\end{equation}
for some constant $\varepsilon > 0$. Then for every $z \in U_{j,k} \cap S_n$ there exists a right triangle $T \subset U_{j, k} \cap S_n$, with one of the vertices at $z$, a horizontal leg of length $c_1 D^{(n)}_j$ and a vertical leg of length $c_2 e^{j\delta}$, containing at least one element of $\tilde \KK_{j,k}^{(n)}$, where the constants $c_1, c_2 > 0$ depend only on $a, \varepsilon$ and $q$.
\end{lem}
The proof of Lemma~\ref{lem:triangle}, using \eqref{eq:<A<} and \eqref{eq:B}, is an elementary, but a bit tedious exercise  and is left to the reader. 

The next lemma provides basic estimates of the derivative of the inverse branches of $E_{\lambda_{N+n}} \circ \cdots \circ E_{\lambda_N}$. 

\begin{lem}\label{lem:branch} For every $n \in \N$ and $j_0, \ldots, j_n \in \Z$, $k_0, \ldots, k_n \in \Z$, such that $K^{(0)}_{j_0,k_0} \cap S_0 \ne \emptyset$, $K^{(1)}_{j_1,k_1} \in \KK^{(1)}_{j_0, k_0}, \ldots, K^{(n)}_{j_n,k_n} \in \KK^{(n)}_{j_{n-1}, k_{n-1}}$, the branch
\[
g^{(0)}_{k_0} \circ \cdots \circ g^{(n-1)}_{k_{n-1}}
\]
is defined on $K^{(n)}_{j_n,k_n}$, for some extensions of the branches from \eqref{eq:g_k}, with the distortion bounded by a constant independent of $n$ and $j_0, \ldots, j_n$, $k_0, \ldots, k_n$. Moreover,
\[
\frac{c^{-n}}{e^{\delta (j_0 + \cdots + j_{n-1})}} < \left|(g^{(0)}_{k_0} \circ \cdots \circ g^{(n-1)}_{k_{n-1}})'|_{K^{(n)}_{j_n,k_n}}\right| < \frac{c^n}{e^{\delta (j_0 + \cdots + j_{n-1})}} < \frac{1}{2^n}
\]
for some constant $c > 0$.
\end{lem}
\begin{proof}
Take $j_0, \ldots, j_n$, $k_0, \ldots, k_n$ as in the lemma. By assumption, 
\begin{equation}\label{eq:KUS}
K^{(m-1)}_{j_{m-1},k_{m-1}} \cap S_{m-1} \ne \emptyset, \qquad K^{(m)}_{j_m,k_m} \cap U_{j_{m-1},k_{m-1}} \ne \emptyset
\end{equation}
for $m = 1, \ldots, n$. Let 
\[
d_0 = \sqrt{\pi^2 + \delta^2}
\]
be the diameter of the sets $K \in \bigcup_{n=0}^\infty\KK^{(n)}$. The first assertion of \eqref{eq:KUS} together with \eqref{eq:<j<} and \eqref{eq:ass-a>} implies 
\begin{equation}\label{eq:e^a}
e^{j_{m-1}\delta} \ge e^{-\delta} a_{N + m-1} \ge e^{-\delta} a > 2d_0+2,
\end{equation}
if $a$ is chosen sufficiently large. Hence, 
\begin{equation}\label{eq:U>1}
U_{j_{m-1},k_{m-1}} \subset \{z \in \C: |z| \ge e^{-\delta} a\}
\end{equation}
and the branch $g_{k_{m-1}}^{(m-1)}$ on $U_{j_{m-1},k_{m-1}}$ can be extended to
\[
\hat U_{j_{m-1},k_{m-1}} = \{z \in \C: \dist(z, U_{j_{m-1},k_{m-1}}) <  2d_0\}.
\]
Let
\[
V_{m,m} = K^{(m)}_{j_m,k_m}, \qquad  V_{m,s} = g^{(m)}_{k_m} \circ \cdots \circ g^{(s-1)}_{k_{s-1}}(K^{(s)}_{j_s,k_s})
\]
for $m = 0, \ldots, n$, $s = m+1, \ldots, n$. Now we show, by backward induction on $m$, that 
\begin{equation}\label{eq:ind}
\begin{aligned}
&V_{m,s} \text{ are well-defined} &&\text{for }  s = m, \ldots, n,\\
&\diam V_{m,s} \le \frac{d_0}{2^{s-m}}&&\text{for }  s = m, \ldots, n,\\
&V_{m,s} \subset \hat U_{j_{m-1},k_{m-1}}  &&\text{for }  s = m, \ldots, n,\\
&V_{m, s} \cap V_{m, s+1} \ne \emptyset   &&\text{for }  s = m, \ldots, n-1.
\end{aligned}
\end{equation}
For $m = n$, \eqref{eq:ind} follows from \eqref{eq:KUS}. Suppose, by induction, that \eqref{eq:ind} holds for some $1 \le m \leq n$. Then $V_{m-1,s} = g^{(m-1)}_{k_{m-1}}(V_{m,s})$ for $s = m, \ldots, n$ are well defined. Take $s \in \{m-1 , \ldots, n\}$. By \eqref{eq:e^a} and the fourth assertion of \eqref{eq:ind}, 
\[
V_{m,s} \subset \Big\{z \in \C: |z| \ge e^{-\delta} a - d_0\Big(1 + \cdots + \frac{1}{2^m}\Big)\Big\} \subset \{z \in \C: |z| \ge 2\},
\]
so 
\begin{equation}\label{eq:diamV}
\diam V_{m-1,s} \le \sup_{V_{m,s}}|(g^{(m-1)}_{k_{m-1}})'| \diam V_{m,s} \le \frac{\diam V_{m,s}}{2} < \frac{d_0}{2^{s-m+1}}.
\end{equation}
By \eqref{eq:KUS} and the fourth assertion of \eqref{eq:ind}, $V_{m-1, s} \cap V_{m-1, s+1} \ne \emptyset$ for $s = m-1, \ldots, n-1$. Hence, by \eqref{eq:KUS}, to have $V_{m-1,s} \subset \hat U_{j_{m-2},k_{m-2}}$ for $s = m-1 , \ldots, n$, it is enough to check that 
\[
\diam V_{m-1,m-1} + \cdots + \diam V_{m-1, n} < 2d_0,
\]
which follows from \eqref{eq:diamV}. This ends the inductive proof of \eqref{eq:ind}.

By \eqref{eq:ind} for $m = 0, s = n$ we conclude that the branch $g^{(0)}_{k_0} \circ \cdots \circ g^{(n-1)}_{k_{n-1}}$ is defined on $K^{(n)}_{j_n,k_n}$. The distortion of the branch is estimated in a standard way. By \eqref{eq:e^a}, \eqref{eq:U>1} and \eqref{eq:ind}, for $z_1, z_2 \in V_{m, n}$ we have
\[
\frac{|(g^{(m-1)}_{k_{m-1}})'(z_1)|}{|(g^{(m-1)}_{k_{m-1}})'(z_2)|} = \frac{|z_2|}{|z_1|} \le 1 + \frac{|z_1 - z_2|}{|z_1|} \le 1 + \frac{\diam V_{m, n}}{e^{-\delta}a - 2 d_0} \le 1 + \frac{d_0}{2^{n-m+1}}. 
\]
Hence, for $z_1, z_2 \in V_{m, n}$,
\[
\frac{|(g^{(0)}_{k_0} \circ \cdots \circ g^{(n-1)}_{k_{n-1}})'(z_1)|}{|(g^{(0)}_{k_0} \circ \cdots \circ g^{(n-1)}_{k_{n-1}})'(z_2)|} \le \prod_{m=1}^n \left(1 + \frac{d_0}{2^{m+1}}\right)\le \exp \left(\sum_{m=1}^n \frac{d_0}{2^{m+1}}\right) < e^{d_0/2},
\]
so the distortion of the branch is universally bounded. Finally, \eqref{eq:e^a} and the third assertion of \eqref{eq:ind} give
\begin{align*}
\frac{c^{-n}}{e^{ (j_0 + \cdots + j_{n-1})\delta}} \le \prod_{m=1}^n\frac{1}{e^{(j_{m-1}+1)\delta} + 2d_0} &< \left|(g^{(0)}_{k_0} \circ \cdots \circ g^{(n-1)}_{k_{n-1}})'|_{K^{(n)}_{j_n,k_n}}\right|\\
&< \prod_{m=1}^n\frac{1}{e^{j_{m-1}\delta} - d_0} \le \frac{c^n}{e^{ (j_0 + \cdots + j_{n-1})\delta}} < \frac{1}{Q^n}
\end{align*}
for $c= \max(e^\delta(1 + 2d_0/a), 1/(1 - 2d_0 e^{\delta}/ a))$, $Q = e^{-\delta} a/c$. Choosing $a$ sufficiently large, we can assume $Q \ge 2$. 
\end{proof}

\section{Proof of Theorem~\ref{thm:B} -- estimate from above}\label{sec:proof-B-above}

In this section we prove the upper estimate in Theorem~\ref{thm:B}. First, we will do this under an additional technical assumption
\begin{equation}\label{eq:<B<}
|A_n + e^{j\delta}| > \varepsilon  e^{j\delta}, \qquad |B_n - e^{j\delta}| > \varepsilon e^{j\delta} 
\end{equation}
for every $n \ge 0$, $j\in \Z$ and some constant $\varepsilon > 0$. In the last subsection we will show how to reduce the general situation to this case.

\subsection{\boldmath Construction of the measure $\mu$}

Take $j_0, k_0 \in \Z$ such that $J \cap K^{(0)}_{j_0,k_0}\ne \emptyset$. In particular, we have $K^{(0)}_{j_0,k_0} \cap S_0 \ne \emptyset$. By Lemma~\ref{lem:branch}, we can define families $\FF^{(n)}$, $n \ge 0$, setting
\[
\FF^{(0)} = \{K_{j_0, k_0}\} \qquad \text{for} \quad K_{j_0, k_0} = K^{(0)}_{j_0, k_0}
\]
and 
\begin{align*}
\FF^{(n)} = \{&K_{j_0, k_0, \ldots, j_n, k_n} = g^{(0)}_{k_0} \circ \cdots \circ g^{(n-1)}_{k_{n-1}}(K^{(n)}_{j_n, k_n}):\\
&K^{(1)}_{j_1,k_1} \in \KK^{(1)}_{j_0, k_0}, \ldots, K^{(n)}_{j_n,k_n} \in \KK^{(n)}_{j_{n-1}, k_{n-1}}, \; j_1, \ldots, j_n \in \Z, \; k_1, \ldots, k_n \in \Z\}
\end{align*}
for $n \in \N$.

Since, for given $n$, the sets $K^{(n)}_{j, k} \in \KK^{(n)}$ are pairwise disjoint, the sets $K_{j_0, k_0, \ldots, j_n, k_n} \in \FF^{(n)}$ are pairwise disjoint. Moreover, for every set $K_{j_0, k_0, \ldots, j_n, k_n} \in \FF^{(n)}$ and $j_{n+1}, k_{n+1} \in \Z$ we have
\begin{equation}\label{eq:cap}
\begin{aligned}
K_{j_0, k_0, \ldots, j_{n+1}, k_{n+1}} \in \FF^{(n+1)} &\iff  K^{(n+1)}_{j_{n+1},k_{n+1}} \in \KK^{(n+1)}_{j_n, k_n}\\ &\iff K_{j_0, k_0, \ldots, j_{n+1}, k_{n+1}} \cap K_{j_0, k_0, \ldots, j_n, k_n} \ne \emptyset.
\end{aligned}
\end{equation}
For $K_{j_0, k_0, \ldots, j_n, k_n} \in \FF^{(n)}$ let
\[
N_{j_0, k_0, \ldots, j_n, k_n} = \#\{(j_{n+1}, k_{n+1}) : K_{j_0, k_0, \ldots, j_{n+1}, k_{n+1}} \in \FF^{(n+1)}\} = \#\KK^{(n+1)}_{j_n, k_n}.
\]
By \eqref{eq:<j<}, \eqref{eq:ass-a>} and Lemma~\ref{lem:D},  
\begin{equation}\label{eq:N<}
0 < N_{j_0, k_0, \ldots, j_n, k_n} \leq c_1 D^{(n+1)}_{j_n} e^{j_n\delta},
\end{equation}
for some constant $c_1 > 0$, and if $K_{j_0, k_0, \ldots, j_{n+1}, k_{n+1}} \in \FF^{(n+1)}$, then
\begin{equation}\label{eq:diamK'}
\frac{\diam K_{j_0, k_0, \ldots, j_{n+1}, k_{n+1}}}{\diam K_{j_0, k_0, \ldots, j_n, k_n}} \le c \sup_{U_{j_n, k_n}}|(g^{(n)}_{k_n})'| < \frac{ce^\delta}{a} < \frac{1}{2}
\end{equation}
for a constant $c > 0$, provided $a$ is chosen sufficiently large.

Let
\[
K_\infty = \overline{\bigcap_{n=0}^\infty \bigcup\FF^{(n)}}.
\]
In the same way as for Lemma~\ref{lem:cover}, we show
\begin{equation}
\label{lem:JinF}
J \cap K^{(0)}_{j_0,k_0} \subset K_\infty.
\end{equation}

For every $n \ge 0$ and $K_{j_0, k_0, \ldots, j_n, k_n}\in \FF^{(n)}$ choose a point 
\[
z_{j_0, k_0, \ldots, j_n, k_n} \in K_{j_0, k_0, \ldots, j_n, k_n}
\] 
and note that by \eqref{eq:cap} and \eqref{eq:diamK'}, if $K_{j_0, k_0, \ldots, j_m, k_m} \in \FF^{(m)}$ for some $m > n$, then
\begin{equation}\label{eq:u-u<}
|z_{j_0, k_0, \ldots, j_m, k_m} - z_{j_0, k_0, \ldots, j_n, k_n}| < \Big(1 + \cdots + \frac{1}{2^{m-n}}\Big)\diam K_{j_0, k_0, \ldots, j_n, k_n} < \frac{d_0}{2^{n-1}}.
\end{equation}
Define a sequence of Borel probability measures $\mu_n$, $n \ge 0$, setting
\begin{align*}
\mu_0 &= \nu_{z_{j_0, k_0}},\\
\mu_{n+1} &= \sum_{K_{j_0, k_0, \ldots, j_n, k_n} \in \FF^{(n)}} \; \sum_{(j_{n+1}, k_{n+1}) : K_{j_0, k_0, \ldots, j_{n+1}, k_{n+1}} \in \FF^{(n+1)}} \frac{\nu_{z_{j_0, k_0, \ldots, j_{n+1}, k_{n+1}}}}{N_{j_0, k_0} \cdots N_{j_0, k_0, \ldots, j_n, k_n}},
\end{align*}
where $\nu_z$ denotes the Dirac measure at $z$. By definition, for every $K_{j_0, k_0, \ldots, j_n, k_n} \in \FF^{(n)}$,
\[
\mu_{n+1}(\{z_{j_0, k_0, \ldots, j_n, k_n, j_{n+1}, k_{n+1}}: K_{j_0, k_0, \ldots, j_n, k_n, j_{n+1}, k_{n+1}} \in \FF^{(n+1)}\} = \mu_n(\{z_{j_0, k_0, \ldots, j_n, k_n}\}),
\]
so by induction, using \eqref{eq:u-u<}, we obtain
\begin{equation}\label{eq:mu_n+1}
\begin{aligned}
\mu_m(\D(z_{j_0, k_0, \ldots, j_n, k_n}, 2 \diam K_{j_0, k_0, \ldots, j_n, k_n}) &\ge \mu_n(\{z_{j_0, k_0, \ldots, j_n, k_n}\})\\
&= \frac{1}{N_{j_0, k_0} \cdots N_{j_0, k_0, \ldots, j_{n-1}, k_{n-1}}}
\end{aligned}
\end{equation}
for every $m \ge n$. By \eqref{eq:u-u<},
\[
\supp \mu_n = \{z_{j_0, k_0, \ldots, j_n, k_n}: K_{j_0, k_0, \ldots, j_n, k_n} \in \FF^{(n)}\} \subset \D(z_{j_0, k_0}, 2d_0).
\]
Hence, the sequence $\mu_n$ converges weakly along a subsequence to a Borel probability measure $\mu$ with support in $\overline{\D(z_{j_0, k_0}, 2d_0)}$. 

Take $K_{j_0, k_0, \ldots, j_n, k_n} \in \FF^{(n)}$. By \eqref{eq:N<} and \eqref{eq:mu_n+1},
\begin{equation}\label{eq:mu>}
\mu(\hat K_{j_0, k_0, \ldots, j_n, k_n}) \ge  \frac{c_1^{-n}}{D^{(1)}_{j_0} \cdots D^{(n)}_{j_{n-1}} e^{(j_0 + \cdots +j_{n-1})\delta}}
\end{equation} 
for
\[
\hat K_{j_0, k_0, \ldots, j_n, k_n} = \{z \in \C: \dist(z, K_{j_0, k_0, \ldots, j_n, k_n}) \le 2 \diam K_{j_0, k_0, \ldots, j_n, k_n})\}.
\]
 
\subsection{\boldmath Estimate of the local dimension of $\mu$}

Since every point in the support of $\mu$ is a limit of points from $\supp \mu_{n_s}$ for some $n_s \to \infty$, taking a suitable subsequence and using \eqref{eq:u-u<} we obtain
\begin{align*}
\supp \mu \subset \{&z \in \C: z = \lim_{n\to\infty} z_{j_0, k_0, \ldots, j_n, k_n}, \text{ where } j_1, k_1, j_2, k_2, \ldots \in \Z\\
&\text{and } K_{j_0, k_0, \ldots, j_n, k_n} \in \FF^{(n)} \text{ for every } n \ge 0\}.
\end{align*}
The same argument show
\begin{equation}\label{eq:Kinmu}
K_\infty \subset \supp \mu.
\end{equation}

Take $z = \lim_{n\to\infty} z_{j_0, k_0, \ldots, j_n, k_n}\in \supp \mu$, where $K_{j_0, k_0, \ldots, j_n, k_n} \in \FF^{(n)}$ for every $n \ge 0$. For simplicity, denote
\[
d_n = \diam K_{j_0, k_0, \ldots, j_n, k_n}, \qquad z_n = z_{j_0, k_0, \ldots, j_n, k_n}.
\]
By \eqref{eq:u-u<}, we have 
\begin{equation}\label{eq:z-u}
|z - z_n| \le 2 d_n.
\end{equation}
Let
\[
r_n = C d_n
\]
for a large constant $C > 0$. Note that by \eqref{eq:diamK'}, the sequence $r_n$ is strictly decreasing to $0$.

Now we will estimate $\mu(\D(z, r))$ for a small $r$. Let $n$ be such that 
\[
r_{n+1} \le r < r_n
\]
and let
\[
R = \frac{r}{\sqrt{C} d_{n+1}}.
\]
Note that if $r$ varies in $[r_{n+1}, r_n)$, then $R$ varies in $[R_-^{(n)}, R_+^{(n)})$ for
\[
R_-^{(n)} = \sqrt{C}, \qquad
R_+^{(n)} = \sqrt{C}\frac{d_n}{d_{n+1}}.
\]
By Lemma~\ref{lem:branch}, 
\begin{equation}\label{eq:<r<}
\frac{\sqrt{C}}{c_2^{n+1}} \frac{R}{e^{(j_0 + \cdots +j_n)\delta}} < r < c_2^{n+1}\sqrt{C} \frac{R}{e^{(j_0 + \cdots +j_n)\delta}}
\end{equation}
and
\begin{equation}\label{eq:R>1}
\frac{\sqrt{C}}{c_2} e^{j_n\delta} \le R_+^{(n)} \leq  c_2\sqrt{C} e^{j_n\delta}
\end{equation}
for some constant $c_2 > 0$. Enlarging additionally $C$, by Lemma~\ref{lem:D} and \eqref{eq:R>1} we can assume 
\begin{equation}\label{eq:D>R}
\frac{c_3}{\sqrt{C}} R^{(n)}_- < D^{(n+1)}_{j_n} <  R^{(n)}_+
\end{equation}
for some constant $c_3 > 0$. 

Let
\[
w = E_{\lambda_{N+n}} \circ \cdots \circ E_{\lambda_N}(z_{n+1}).
\]
By definition, $w \in K^{(n+1)}_{j_{n+1},k_{n+1}} \in \KK^{(n+1)}_{j_n, k_n}$.
Take $K^{(n+1)}_{j'_{n+1}, k'_{n+1}} \in \KK^{(n+1)}_{j_n,k_n}$ such that $K^{(n+1)}_{j'_{n+1}, k'_{n+1}} \subset \D(w, R)$. By \eqref{eq:cap}, $K_{j_0, k_0, \ldots, j_n, k_n, j'_{n+1}, k'_{n+1}}\in \FF^{(n+1)}$ and by Lemma~\ref{lem:branch}, 
there exists a constant $c_4 > 0$ such that
\[
|z_{j_0, k_0, \ldots, j_n, k_n, j'_{n+1}, k'_{n+1}} - z_{n+1}| < c_4 R d_{n+1}, \qquad \diam K_{j_0, k_0, \ldots, j_n, k_n, j'_{n+1}, k'_{n+1}} < c_4 d_{n+1}. 
\]
Using this together with \eqref{eq:z-u} we obtain
\begin{equation}
\label{eq:hatKin}
\begin{aligned}
\hat K_{j_0, k_0, \ldots, j_n, k_n, j'_{n+1}, k'_{n+1}} &\subset 
\D(z, (c_4R + c_4 +2) d_{n+1})\\ 
&= \D\Big(z, \frac{c_4}{\sqrt{C}}r + \frac{c_4+2}{C}r_{n+1} \Big)\\
&\subset \D\Big(z, \Big(\frac{c_4}{\sqrt{C}} + \frac{c_4+2}{C}\Big)r\Big) \subset \D(z,r).
\end{aligned}
\end{equation}if $C$ is chosen sufficiently large. 

By \eqref{eq:<B<} and Lemma~\ref{lem:D}, there exist $u \in K^{(n+1)}_{j_{n+1},k_{n+1}} \cap U_{j_n, k_n} \cap S_{n+1}$ and a right triangle $T \subset U_{j_n, k_n} \cap S_{n+1}$, with one of the vertices at $u$, a horizontal leg of length $c D^{(n+1)}_{j_n}$ and a vertical leg of length $c' e^{j_n\delta}$, for some constants $c, c' > 0$, containing at least one element of $\KK_{j_n,k_n}^{(n+1)}$. Note also that Lemma~\ref{lem:branch} implies that if $K^{(n+1)}_{j'_{n+1}, k'_{n+1}}, K^{(n+1)}_{j''_{n+1}, k''_{n+1}} \in \KK_{j_n,k_n}^{(n+1)}$ and $\dist(K^{(n+1)}_{j'_{n+1}, k'_{n+1}}, K^{(n+1)}_{j''_{n+1}, k''_{n+1}}) > c_5$ for a sufficiently large constant $c_5 > 0$, then $\hat K_{j_0, k_0, \ldots, j_n, k_n, j'_{n+1}, k'_{n+1}}$ and $\hat K_{j_0, k_0, \ldots, j_n, k_n, j''_{n+1}, k''_{n+1}}$ are disjoint. Using these facts and noting that $R \ge \sqrt{C}$ for a large $C$, we show by elementary geometry considerations that $\D(w, R)$ contains at least $M$ sets $K^{(n+1)}_{j'_{n+1}, k'_{n+1}} \in \KK_{j_n,k_n}^{(n+1)}$, such that $\hat K_{j_0, k_0, \ldots, j_n, k_n, j'_{n+1}, k'_{n+1}}$ are pairwise disjoint, where
\[
M =
\begin{cases}
c_6 R^2& \text{if } R \le D^{(n+1)}_{j_n}\\
c_6 D^{(n+1)}_{j_n} R& \text{if } R > D^{(n+1)}_{j_n}
\end{cases}
\]
for some constant $c_6 > 0$. By \eqref{eq:mu>} and \eqref{eq:hatKin},
\[
\mu(\D(z, r)) \ge 
\begin{cases} 
\frac{c_6c_1^{-(n+1)} R^2}{D^{(1)}_{j_0} \cdots D^{(n+1)}_{j_n}e^{(j_0 + \cdots +j_n)\delta}}& \text{if } R \le D^{(n+1)}_{j_n}\\
\frac{c_6 c_1^{-(n+1)} R}{D^{(1)}_{j_0} \cdots D^{(n)}_{j_{n-1}} e^{(j_0 + \cdots +j_n)\delta}}& \text{if } R > D^{(n+1)}_{j_n}
\end{cases},
\]
so by \eqref{eq:<r<},
\begin{equation}\label{eq:log<}
\frac{\log \mu(\D(z, r))}{\log r}\le 1 + h_n(R),
\end{equation}
where
\begin{equation}\label{eq:h_n}
h_n(x) = 
\begin{cases}\frac{\log(D^{(1)}_{j_0} \cdots D^{(n+1)}_{j_n}) - \log 
x+ c_7 n}{(j_0 + \cdots +j_n)\delta - \log x - c_7 n}& \text{if } x \le D^{(n+1)}_{j_n}\\
\frac{\log(D^{(1)}_{j_0} \cdots D^{(n)}_{j_{n-1}}) + c_7 n}{(j_0 + \cdots +j_n)\delta - \log x - c_7 n}& \text{if } x > D^{(n+1)}_{j_n} 
\end{cases}
\end{equation}
for $x \in [R^{(n)}_-, R^{(n)}_+)$ and some constant $c_7 > 0$, which can be chosen arbitrary large. Note that 
by \eqref{eq:<j<} and \eqref{eq:ass-infty1}, we have
\begin{equation}\label{eq:ass-infty}
\frac{j_0 + \cdots +j_n}{n} \to \infty \qquad \text{as } n \to \infty.
\end{equation}
Together with \eqref{eq:R>1}, this implies that the denominators in \eqref{eq:h_n} are positive for large $n$, so $h_n$ is well-defined.

Now we estimate the infimum and supremum of the function $h_n$.

\begin{lem}\label{lem:h}
\begin{align*}
&\lim_{n\to\infty}\left|\inf_{[R^{(n)}_-, R^{(n)}_+)} h_n -  \frac{\log(D^{(1)}_{j_0} \cdots D^{(n)}_{j_{n-1}})}{(j_0 + \cdots +j_n)\delta - \log D^{(n+1)}_{j_n}}\right| = 0,\\
&\lim_{n\to\infty}\left|\sup_{[R^{(n)}_-, R^{(n)}_+)} h_n - \max\left(\frac{\log(D^{(1)}_{j_0} \cdots D^{(n)}_{j_{n-1}})}{(j_0 + \cdots +j_{n-1})\delta}, \frac{\log(D^{(1)}_{j_0} \cdots D^{(n+1)}_{j_n})}{(j_0 + \cdots +j_n)\delta}\right)\right| = 0.
\end{align*}
\end{lem}
\begin{proof}

We can write
\[
h_n(x) = 
\begin{cases}h^{(n)}_1(x) + h^{(n)}_2(x) &\text{if } x \le D^{(n+1)}_{j_n}\\
h^{(n)}_3(x) &\text{if } x > D^{(n+1)}_{j_n}
\end{cases},
\]
where
\begin{align*}
h^{(n)}_1(x) &= 1 +\frac{\log(D^{(1)}_{j_0} \cdots D^{(n+1)}_{j_n}) - (j_0 + \cdots +j_n)\delta - c_8n}{(j_0 + \cdots +j_n)\delta - \log x - c_7 n},\\
h^{(n)}_2(x) &= \frac{(2c_7 + c_8)n}{(j_0 + \cdots +j_n)\delta - \log x - c_7 n}\\
h^{(n)}_3(x) &= \frac{\log(D^{(1)}_{j_0} \cdots D^{(n)}_{j_{n-1}}) + c_7 n}{(j_0 + \cdots +j_n)\delta - \log x - c_7 n} 
\end{align*}
for $x \in [R^{(n)}_-, R^{(n)}_+)$ and a large constant $c_8 > 0$. Let
\[
\varepsilon_n = \sup_{[R^{(n)}_-, R^{(n)}_+)} |h^{(n)}_2|
\]
and note that by \eqref{eq:R>1} and \eqref{eq:ass-infty}, we have $\varepsilon_n\to 0$ as $n \to \infty$. By Lemma~\ref{lem:D}, $h^{(n)}_1$ is decreasing and $h^{(n)}_3$ is increasing, if $c_7$ and $c_8$ are chosen sufficiently large. This together with \eqref{eq:D>R} implies that if $D^{(n+1)}_{j_n} > R^{(n)}_-$, then
\[
\left|\inf_{[R^{(n)}_-, R^{(n)}_+)} h_n - h^{(n)}_1(D^{(n+1)}_{j_n})\right| \le \varepsilon_n, \qquad 
\left|\sup_{[R^{(n)}_-, R^{(n)}_+)} h_n - \max\left(h^{(n)}_1(R^{(n)}_-), h^{(n)}_3(R^{(n)}_+)\right)\right| \le \varepsilon_n,
\]
and if $D^{(n+1)}_{j_n} < R^{(n)}_-$, then
\[
\left|\inf_{[R^{(n)}_-, R^{(n)}_+)} h_n - h^{(n)}_1(R^{(n)}_-)\right| \le  \varepsilon_n, \qquad
\left|\sup_{[R^{(n)}_-, R^{(n)}_+)} h_n - h^{(n)}_3(R^{(n)}_+)\right| \le  \varepsilon_n.
\]
Furthermore, using \eqref{eq:R>1}, \eqref{eq:D>R} and \eqref{eq:ass-infty} we obtain
\begin{align*}
&\left|h^{(n)}_1(D^{(n+1)}_{j_n}) - \frac{\log(D^{(1)}_{j_0} \cdots D^{(n)}_{j_{n-1}})}{(j_0 + \cdots +j_n)\delta - \log D^{(n+1)}_{j_n}}\right| \to 0,\\
&\left|h^{(n)}_1(R^{(n)}_-) - \frac{\log(D^{(1)}_{j_0} \cdots D^{(n+1)}_{j_n})}{(j_0 + \cdots +j_n)\delta}\right| \to 0,\\
&\left|h^{(n)}_3(R^{(n)}_+) - \frac{\log(D^{(1)}_{j_0} \cdots D^{(n)}_{j_{n-1}})}{(j_0 + \cdots +j_{n-1})\delta}\right| \to 0
\end{align*}
and
\[
|h^{(n)}_1(D^{(n+1)}_{j_n}) - h^{(n)}_1(R^{(n)}_-)| \to 0 \qquad\text{if } D^{(n+1)}_{j_n} < R^{(n)}_-
\]
as $n \to \infty$. This proves the lemma.
\end{proof}

\subsection{Conclusion}

By \eqref{lem:JinF}, \eqref{eq:Kinmu}, \eqref{eq:log<} and Lemma~\ref{lem:h}, for every $j_0, k_0 \in \Z$ such that $J \cap K^{(0)}_{j_0,k_0} \ne \emptyset$ and every $z \in J \cap K^{(0)}_{j_0,k_0}$ there exist $j_1, k_1, j_2, k_2, \ldots \in \Z$ with $z  = \lim_{n\to\infty} z_{j_0, k_0, \ldots, j_n, k_n}$ and  
\begin{align*}
\liminf_{r \to 0} \frac{\log \mu(\D(z, r))}{\log r}&\le
1 + \liminf_{n \to \infty}\frac{\log(D^{(1)}_{j_0} \cdots D^{(n)}_{j_{n-1}})}{(j_0 + \cdots +j_n)\delta - \log D^{(n+1)}_{j_n}},\\
\limsup_{r \to 0} \frac{\log \mu(\D(z, r))}{\log r}&\le 
1 + \limsup_{n \to \infty} \frac{\log(D^{(1)}_{j_0} \cdots D^{(n+1)}_{j_n})}{(j_0 + \cdots +j_n)\delta}.
\end{align*}
This together with Lemma~\ref{lem:D}, \eqref{eq:<j<} and \eqref{eq:ass-infty}  implies
\begin{equation}\label{eq:log_mu<}
\begin{aligned}
\liminf_{r \to 0} \frac{\log \mu(\D(z, r))}{\log r}&\le
1 + \liminf_{n \to \infty}\Phi_n(\delta j_0, \ldots,\delta j_n),\\
\limsup_{r \to 0} \frac{\log \mu(\D(z, r))}{\log r}&\le 
1 + \limsup_{n \to \infty} \Psi_n(\delta j_0, \ldots,\delta j_n),
\end{aligned}
\end{equation}
where 
\begin{equation}\label{eq:PhiPsi}
\begin{aligned}
\Phi_n(x_0, \ldots, x_n) &= \frac{\min(\log \Delta_{N+1}, x_0) + \cdots + \min(\log \Delta_{N+n}, x_{n-1})}{x_0 + \cdots +x_n - \min(\log \Delta_{N+n+1}, x_n)},\\
\Psi_n(x_0, \ldots, x_n) &= \frac{\min(\log \Delta_{N+1}, x_0) + \cdots + \min(\log \Delta_{N+n+1}, x_n)}{x_0 + \cdots +x_n}
\end{aligned}
\end{equation}
for $x_0 \in [\log a_N, \log b_N]$, $x_1 \in [\log a_{N+1}, \log b_{N+1}], \ldots$. By the standard dimension estimates (see e.g.~\cite{mattila,PUbook}), \eqref{eq:log_mu<} gives
\begin{align*}
\dim_H J \le 1 + \sup_x\liminf_{n \to \infty}\Phi_n(x_0, \ldots, x_n),\\
\dim_P J \le 1 + \sup_x\limsup_{n \to \infty}\Psi_n(x_0, \ldots, x_n)
\end{align*}
for $x = (x_0, x_1, \ldots) \in [\log a_N, \log b_N] \times [\log a_{N+1}, \log b_{N+1}] \times \cdots$. Together with \eqref{eq:dim}, this proves the upper estimate in Theorem~\ref{thm:B}.

\subsection{General case}

Suppose now that the assumption \eqref{eq:<B<} does not hold. For $n \ge 0$, 
if $A_n < 0$, then let $\alpha_n \in \Z$ be such that 
\[
-e^{(\alpha_n + 1)\delta} \le A_n < - e^{\alpha_n\delta}.
\]
Similarly, if $B_n > 0$, then let $\beta_n \in \Z$ be such that 
\[
e^{\beta_n\delta} \le B_n < e^{(\beta_n + 1)\delta}.
\]
Set $a'_m = a_m$, $b'_m = b_m$ for $1 \le m < N$ and
\[
a'_{N+n} = |\lambda_{N+n}| e^{A'_n}, \qquad b'_{N+n} = |\lambda_{N+n}| e^{B'_n}
\] 
for $ n \ge 0$, where
\begin{align*}
A'_n &= 
\begin{cases}
-e^{(\alpha_n + 3/2)\delta} 
&\text{if } -e^{(\alpha_n + 1)\delta} \le A_n < -e^{\alpha_n\delta} \\
A_n &\text{if } A_n \ge 0
\end{cases},\\
B'_n &= 
\begin{cases}
e^{(\beta_n+ 3/2)\delta} 
&\text{if }  e^{\beta_n\delta} \le B_n < e^{(\beta_n + 1)\delta}\\
B_n &\text{if } B_n \le 0
\end{cases}.
\end{align*}
By definition, 
\[
-\frac{A'_n}{e^{j\delta}} \le e^{-\delta/2} \quad \text{or} \quad -\frac{A'_n}{e^{j\delta}} \ge e^{\delta/2}, \qquad 
\frac{B'_n}{e^{j\delta}} \le e^{-\delta/2} \quad \text{or} \quad \frac{B'_n}{e^{j\delta}} \ge e^{\delta/2}
\]
for every $n \ge 0$, $j \in \Z$, so the condition \eqref{eq:<B<} is satisfied for $A'_n, B'_n$ instead of $A_n, B_n$. Therefore, we can repeat the proof contained in this section, replacing $a_n$ by $a'_n$ and $b_n$ by $b'_n$, respectively. Since
\[
e^{-3\delta/2} A_n \le A'_n \le A_n, \qquad B_n \le B'_n \le e^{3\delta/2} B_n
\]
for every $n \ge 0$, this replacement does not spoil the assumptions of Theorem~\ref{thm:B}. Moreover, the values of $\log D^{(n+1)}_{j_n}$, $n \ge 0$, change at most by an additive constant. Hence, using \eqref{eq:ass-infty}, we see that the right sides of the inequalities in \eqref{eq:log_mu<} do not change, so the upper estimates of the Hausdorff and packing dimension of $I_{\aa'}^{\bb'}(E_\ul)$ for $\aa' = (a'_n)_{n=1}^\infty$, $\bb' = (b'_n)_{n=1}^\infty$ are the same as for $I_{\aa}^{\bb}(E_\ul)$. But since $a_n' \le a_n$ and $b_n' \ge b_n$, we have $I_{\aa}^{\bb}(E_\ul) \subset I_{\aa'}^{\bb'}(E_\ul)$, so the estimates are valid also for $I_{\aa}^{\bb}(E_\ul)$.

\section{Proof of Theorem~\ref{thm:B} -- estimate from below}\label{sec:proof-B-below}

\subsection{\boldmath Construction of the measure $\tilde \mu$}

By \eqref{eq:delta<}, we can find $j_0, k_0 \in 2\Z$ such that $K^{(0)}_{j_0,k_0} \subset S_0$. Define families $\tilde \FF^{(n)}$, $n \ge 0$, by
\[
\tilde \FF^{(0)} = \{\tilde K_{j_0, k_0}\} \qquad \text{for} \quad \tilde K_{j_0, k_0} = K^{(0)}_{j_0, k_0}
\]
and 
\begin{align*}
\tilde\FF^{(n)} = \{&\tilde K_{j_0, k_0, \ldots, j_n, k_n} = g^{(0)}_{k_0} \circ \cdots \circ g^{(n-1)}_{k_{n-1}}(K^{(n)}_{j_n, k_n}):\\
&K^{(1)}_{j_1,k_1} \in \tilde \KK^{(1)}_{j_0}, \ldots, K^{(n)}_{j_n,k_n} \in \tilde \KK^{(n)}_{j_{n-1}}, \; j_1, \ldots, j_n \in 2\Z, \; k_1, \ldots, k_n \in 2\Z\}
\end{align*}
for $n \ge 1$. Note that here we consider only even values of $j_0, k_0, j_1, k_1, \ldots$. Obviously, for every $\tilde K_{j_0, k_0, \ldots, j_n, k_n} \in \tilde \FF^{(n)}$ and $j_{n+1}, k_{n+1} \in 2\Z$, 
\begin{equation}
\label{eq:cap'}
\begin{aligned}
&\text{if} \qquad \tilde K_{j_0, k_0, \ldots, j_{n+1}, k_{n+1}} \in \tilde \FF^{(n+1)},\\
&\text{then}\qquad  K^{(n+1)}_{j_{n+1},k_{n+1}} \in \tilde \KK^{(n+1)}_{j_n, k_n} \qquad \text{and} \qquad\tilde K_{j_0, k_0, \ldots, j_{n+1}, k_{n+1}} \subset \tilde K_{j_0, k_0, \ldots, j_n, k_n}.
\end{aligned}
\end{equation}

Moreover, the sets $\overline{\tilde K_{j_0, k_0, \ldots, j_n, k_n}}$ are pairwise disjoint for given $n$. 
Let
\[
\tilde K_\infty = \overline{\bigcap_{n=0}^\infty \bigcup\tilde \FF^{(n)}} = \bigcap_{n=0}^\infty \bigcup\{\overline{\tilde K_{j_0, k_0, \ldots, j_n, k_n}} : \tilde K_{j_0, k_0, \ldots, j_n, k_n} \in \tilde \FF^{(n)}\}.
\]
By definition, we have 
\begin{equation}\label{eq:tildeFinJ}
\tilde K_\infty \subset \overline{J \cap K^{(0)}_{j_0,k_0}} = J \cap \overline{K^{(0)}_{j_0,k_0}},
\end{equation}
since $J$ is closed. 

For $\tilde K_{j_0, k_0, \ldots, j_n, k_n} \in \tilde \FF^{(n)}$ let
\[
\tilde N_{j_0, k_0, \ldots, j_n, k_n} = \#\{(j_{n+1}, k_{n+1}) :\tilde K_{j_0, k_0, \ldots, j_{n+1}, k_{n+1}} \in \tilde \FF^{(n+1)}\}.
\]
By \eqref{eq:<j<}, \eqref{eq:ass-a>} and Lemma~\ref{lem:D},
\begin{equation}\label{eq:tildeN>}
\tilde N_{j_0, k_0, \ldots, j_n, k_n}  \ge \tilde c_1 D^{(n+1)}_{j_n} e^{j_n \delta} > 0
\end{equation}
for a constant $\tilde c_1 > 0$.

For every $n \ge 0$ and  $\tilde K_{j_0, k_0, \ldots, j_n, k_n} \in \tilde \FF^{(n)}$ choose a point 
\[
\tilde z_{j_0, k_0, \ldots, j_n, k_n} \in \tilde K_{j_0, k_0, \ldots, j_n, k_n}
\] 
and define a sequence of Borel probability measures $\tilde \mu_n$, $n \ge 0$ setting
\begin{align*}
\tilde \mu_0 &= \nu_{\tilde z_{j_0, k_0}},\\
\tilde \mu_{n+1} &= \sum_{\tilde K_{j_0, k_0, \ldots, j_n, k_n} \in \tilde \FF^{(n)}} \sum_{(j_{n+1}, k_{n+1}) : \tilde K_{j_0, k_0, \ldots, j_{n+1}, k_{n+1}} \in \tilde \FF^{(n+1)}} 
\frac{\nu_{\tilde z_{j_0, k_0, \ldots, j_{n+1}, k_{n+1}}}}
{\tilde N_{j_0, k_0} \cdots \tilde N_{j_0, k_0, \ldots, j_n, k_n}}.
\end{align*}
By definition, if $\tilde K_{j_0, k_0, \ldots, j_n, k_n} \in \tilde\FF^{(n)}$, then
\begin{equation}\label{eq:mu_n+1'}
\tilde \mu_m(\tilde K_{j_0, k_0, \ldots, j_n, k_n}) = \tilde \mu_n(\tilde K_{j_0, k_0, \ldots, j_n, k_n}) = \frac{1}
{\tilde N_{j_0, k_0} \cdots \tilde N_{j_0, k_0, \ldots, j_{n-1}, k_{n-1}}}
\end{equation}
for every $m \ge n$. Hence, taking a weak limit along a subsequent of $\tilde \mu_n$, we find a Borel probability measure $\tilde \mu$, such that
\begin{equation}\label{eq:Kinmu'}
\supp \tilde \mu  \subset \tilde K_\infty
\end{equation}
and 
\[
\tilde\mu(\tilde K_{j_0, k_0, \ldots, j_n, k_n}) \le \tilde\mu_n(\tilde K_{j_0, k_0, \ldots, j_n, k_n})
\]
for $\tilde K_{j_0, k_0, \ldots, j_n, k_n} \in \tilde \FF^{(n)}$, so by \eqref{eq:tildeN>} and \eqref{eq:mu_n+1'},
\begin{equation}\label{eq:mu<}
\tilde \mu(\tilde K_{j_0, k_0, \ldots, j_n, k_n}) \le \frac{\tilde c_1^{-n}}{D^{(1)}_{j_0} \cdots D^{(n)}_{j_{n-1}} e^{(j_0 + \cdots +j_{n-1})\delta}}.
\end{equation}

\subsection{\boldmath Estimate of the local dimension of $\tilde\mu$}

Take a point $z \in \tilde K_\infty$. Then there exist $j_1, k_1, j_2, k_2, \ldots \in 2\Z$ such that $\tilde K_{j_0, k_0, \ldots, j_n, k_n} \in \tilde \FF^{(n)}$ for every $n \ge 0$,
\[
\tilde K_{j_0, k_0, \ldots, j_n, k_n} \subset \cdots \subset \tilde K_{j_0, k_0}.
\]
Set
\[
\tilde d_n = \diam \tilde K_{j_0, k_0, \ldots, j_n, k_n}, \qquad \tilde z_n = \tilde z_{j_0, k_0, \ldots, j_n, k_n}.
\]
In the same way as for \eqref{eq:diamK'}, we show
\begin{equation}\label{eq:diamK''}
\tilde d_{n+1} < \frac{\tilde d_n}{Q},
\end{equation}
where $Q > 0$ is a constant, which can be chosen arbitrarily large, provided $a$ is big enough. In particular, this implies that $z$ is the unique point of
$\bigcap_{n=0}^\infty \overline{\tilde K_{j_0, k_0, \ldots, j_n, k_n}}$.
Since $z, \tilde z_n \in \overline{\tilde K_{j_0, k_0, \ldots, j_n, k_n}}$, we have 
\begin{equation}\label{eq:z-u'}
|z - \tilde z_n| \le \tilde d_n.
\end{equation}
Let
\[
\tilde r_n = \frac{\tilde d_n}{\tilde C} 
\]
for a large constant $\tilde C > 0$. By \eqref{eq:diamK''}, the sequence $\tilde r_n$ is strictly decreasing to $0$. To estimate $\tilde\mu(\D(z, r))$ for a small $r$, take $n$ such that 
\[
\tilde r_{n+1} \le r < \tilde r_n,
\]
let
\[
\tilde R = \frac{\sqrt{\tilde C}r}{\tilde d_{n+1}}
\]
and note that if $r$ varies in $[\tilde r_{n+1}, \tilde r_n)$, then $\tilde R$ varies in $[\tilde R_-^{(n)}, \tilde R_+^{(n)})$ for
\[
\tilde R_-^{(n)} = \frac{1}{\sqrt{\tilde C}}, \qquad
\tilde R_+^{(n)} = \frac{1}{\sqrt{\tilde C}} \frac{\tilde d_n}{\tilde d_{n+1}}.
\]
By Lemma~\ref{lem:branch}, we have
\begin{equation}\label{eq:<r<'}
\frac{1}{\tilde c_2^{n+1}\sqrt{\tilde C}} \frac{\tilde R}{e^{(j_0 + \cdots +j_n)\delta}} < r < \frac{\tilde c_2^{n+1}}{\sqrt{\tilde C}} \frac{\tilde R}{e^{(j_0 + \cdots +j_n)\delta}}
\end{equation}
and
\begin{equation}\label{eq:R>1'}
\frac{e^{j_n\delta}}{\tilde c_2\sqrt{\tilde C}}  \le \tilde R_+^{(n)} \leq  \frac{\tilde c_2}{\sqrt{\tilde C}} e^{j_n\delta}
\end{equation}
for some constant $\tilde c_2 > 0$. Enlarging additionally $\tilde C$, by Lemma~\ref{lem:D} and \eqref{eq:R>1'} we can assume 
\begin{equation}\label{eq:D>R'}
\tilde R^{(n)}_- < D^{(n+1)}_{j_n} < \tilde c_3 \sqrt{\tilde C} \tilde R^{(n)}_+
\end{equation}
for some constant $\tilde c_3 > 0$. 

Let
\[
\tilde w = E_{\lambda_{N+n}} \circ \cdots \circ E_{\lambda_N}(\tilde z_{n+1}).
\]
Then $\tilde w \in K^{(n+1)}_{j_{n+1},k_{n+1}} \in \tilde\KK^{(n+1)}_{j_n, k_n}$.
Take $j'_1, \ldots, j'_{n+1} \in 2\Z$, $k'_1, \ldots, k'_{n+1} \in 2\Z$ such that $\tilde K_{j_0, k_0, j'_1, k'_1,  \ldots, j'_{n+1}, k'_{n+1}}\in \tilde\FF^{(n+1)}$ and $(j'_1, k'_1,  \ldots, j'_{n+1}, k'_{n+1}) \ne (j_1, k_1,  \ldots, j_{n+1}, k_{n+1})$. Let 
\[
m = \min(\{s \in [1, n+1]: (j'_s, k'_s) \neq (j_s, k_s)\}).
\]
We have $\dist(K^{(m)}_{j_m, k_m}, K^{(m)}_{j'_m, k'_m}) = \delta$, so by Lemma~\ref{lem:branch} and \eqref{eq:cap'}, 
\[
\dist(\tilde z_{n+1}, \tilde K_{j_0, k_0, j'_1, k'_1,  \ldots, j'_{n+1}, k'_{n+1}}) \ge \dist(\tilde K_{j_0, k_0, j_1, k_1,  \ldots, j_m, k_m}, \tilde K_{j_0, k_0, j'_1, k'_1,  \ldots, j'_m, k'_m}) > \tilde c_4 \tilde d_m 
\]
for some constant $\tilde c_4 > 0$. Hence, if $m \le n$, then by \eqref{eq:diamK''} and \eqref{eq:z-u'},
\[
\dist(z, \tilde K_{j_0, k_0, j'_1, k'_1,  \ldots, j'_{n+1}, k'_{n+1}}) \ge 
\Big(\tilde c_4 - \frac{1}{Q}\Big) \tilde d_n > \frac{\tilde c_4 \tilde C r}{2} > r,
\]
if $\tilde C$ and $Q$ are chosen sufficiently large. Consequently, if $\tilde K_{j_0, k_0, j'_1, k'_1,  \ldots, j'_{n+1}, k'_{n+1}}$ intersects $\D(\tilde z, r)$, then $(j'_1, k'_1,  \ldots, j'_n, k'_n) = (j_1, k_1,  \ldots, j_n, k_n)$.  Furthermore, if $\D(\tilde w, \tilde R)$ does not intersect $K^{(n+1)}_{j'_{n+1}, k'_{n+1}}$, then by Lemma~\ref{lem:branch}, 
\[
\dist (\tilde z_{n+1}, \tilde K_{j_0, k_0, j'_1, k'_1,  \ldots, j'_{n+1}, k'_{n+1}}) \ge \tilde c_5 \tilde R \tilde d_{n+1} 
\]
for some constant $\tilde c_5 > 0$, so by \eqref{eq:z-u'},
\[
\dist(z, \tilde K_{j_0, k_0, j'_1, k'_1,  \ldots, j'_{n+1}, k'_{n+1}}) \ge 
(\tilde c_5 \tilde R - 1)\tilde d_{n+1} = \tilde c_5 \sqrt{\tilde C}r - \frac{\tilde r_{n+1}}{\tilde C} \ge \Big(c_5 \sqrt{\tilde C} - \frac{1}{\tilde C}\Big) r > r
\]
provided $\tilde C$ is chosen sufficiently large. We conclude that if $\tilde K_{j_0, k_0, j'_1, k'_1,  \ldots, j'_{n+1}, k'_{n+1}}\in \tilde\FF^{(n+1)}$
and $\tilde K_{j_0, k_0, j'_1, k'_1,  \ldots, j'_{n+1}, k'_{n+1}}$ intersects $\D(\tilde z, r)$, then $(j'_1, k'_1,  \ldots, j'_n, k'_n) = (j_1, k_1,  \ldots, j_n, k_n)$ and $\D(\tilde w, \tilde R)$ intersects $K^{(n+1)}_{j'_{n+1}, k'_{n+1}}$. Note also that in this case we have $K^{(n+1)}_{j'_{n+1}, k'_{n+1}} \in \tilde \KK^{(n+1)}_{j_n, k_n}$, which follows from \eqref{eq:R>1'}, if $\tilde C$ is chosen sufficiently large. Since by Lemma~\ref{lem:D}, the set $\bigcup\tilde \KK^{(n+1)}_{j_n, k_n}$ is contained in a vertical strip of width $D^{(n+1)}_{j_n}$ passing through $\tilde w$, the disc $\D(\tilde w, \tilde R)$ intersects at most $\tilde M$ sets $K^{(n+1)}_{j'_{n+1}, k'_{n+1}} \in \tilde \KK_{j_n,k_n}^{(n+1)}$, where
\[
\tilde M =
\begin{cases}
\tilde c_6 \tilde R^2& \text{if } \tilde R \le D^{(n+1)}_{j_n}\\
\tilde c_6 D^{(n+1)}_{j_n} \tilde R& \text{if } \tilde R > D^{(n+1)}_{j_n}
\end{cases}
\]
for some constant $\tilde c_6 > 0$. By \eqref{eq:mu<},
\[
\tilde \mu(\D(\tilde z, r)) \ge 
\begin{cases} 
\frac{\tilde c_6\tilde c_1^{-(n+1)} \tilde R^2}{D^{(1)}_{j_0} \cdots D^{(n+1)}_{j_n}e^{(j_0 + \cdots +j_n)\delta}}& \text{if } \tilde R \le D^{(n+1)}_{j_n}\\
\frac{\tilde c_6 c_1^{-(n+1)} \tilde R}{D^{(1)}_{j_0} \cdots D^{(n)}_{j_{n-1}} e^{(j_0 + \cdots +j_n)\delta}}& \text{if } \tilde R > D^{(n+1)}_{j_n}
\end{cases},
\]
so by \eqref{eq:<r<'},
\begin{equation}\label{eq:log<'}
\frac{\log \tilde \mu(\D(z, r))}{\log r}\le 1 + \tilde h_n(\tilde R),
\end{equation}
where
\[
\tilde h_n(x) = 
\begin{cases}\frac{\log(D^{(1)}_{j_0} \cdots D^{(n+1)}_{j_n}) - \log 
x+ \tilde c_7 n}{(j_0 + \cdots +j_n)\delta - \log x - \tilde c_7 n}& \text{if } x \le D^{(n+1)}_{j_n}\\
\frac{\log(D^{(1)}_{j_0} \cdots D^{(n)}_{j_{n-1}}) + \tilde c_7 n}{(j_0 + \cdots +j_n)\delta - \log x - \tilde c_7 n}& \text{if } x > D^{(n+1)}_{j_n} 
\end{cases}
\]
for $x \in [\tilde R^{(n)}_-, \tilde R^{(n)}_+)$ and some constant $\tilde c_7 > 0$. 
Note that $j_0, j_1, \ldots$ satisfy \eqref{eq:ass-infty}.
In the same way as for Lemma~\ref{lem:h}, using \eqref{eq:D>R'} instead of \eqref{eq:D>R}, we prove the following.

\begin{lem}\label{lem:h'}
\begin{align*}
&\lim_{n\to\infty}\left|\inf_{[R^{(n)}_-, R^{(n)}_+)} \tilde h_n -  \frac{\log(D^{(1)}_{j_0} \cdots D^{(n)}_{j_{n-1}})}{(j_0 + \cdots +j_n)\delta - \log D^{(n+1)}_{j_n}}\right| = 0,\\
&\lim_{n\to\infty}\left|\sup_{[R^{(n)}_-, R^{(n)}_+)} \tilde h_n - \max\left(\frac{\log(D^{(1)}_{j_0} \cdots D^{(n)}_{j_{n-1}})}{(j_0 + \cdots +j_{n-1})\delta}, \frac{\log(D^{(1)}_{j_0} \cdots D^{(n+1)}_{j_n})}{(j_0 + \cdots +j_n)\delta}\right)\right| = 0.
\end{align*}
\end{lem}

\subsection{Conclusion}

By \eqref{eq:tildeFinJ}, \eqref{eq:Kinmu'}, \eqref{eq:log<'} and Lemma~\ref{lem:h'}, we can find $j_0, k_0 \in 2\Z$ such that for $\tilde \mu$-almost every $z \in J \cap K^{(0)}_{j_0,k_0}$ there exist $j_1, k_1, j_2, k_2, \ldots \in 2\Z$ with $z  = \lim_{n\to\infty} z_{j_0, k_0, \ldots, j_n, k_n}$ and  
\begin{align*}
\liminf_{r \to 0} \frac{\log \tilde \mu(\D(z, r))}{\log r}&\ge
1 + \liminf_{n \to \infty}\frac{\log(D^{(1)}_{j_0} \cdots D^{(n)}_{j_{n-1}})}{(j_0 + \cdots +j_n)\delta - \log D^{(n+1)}_{j_n}},\\
\limsup_{r \to 0} \frac{\log \tilde \mu(\D(z, r))}{\log r}&\ge 
1 + \limsup_{n \to \infty} \frac{\log(D^{(1)}_{j_0} \cdots D^{(n+1)}_{j_n})}{(j_0 + \cdots +j_n)\delta}.
\end{align*}
This together with Lemma~\ref{lem:D}, \eqref{eq:<j<} and \eqref{eq:ass-infty}  implies
\begin{equation}\label{eq:log_mu>}
\begin{aligned}
\liminf_{r \to 0} \frac{\log\tilde  \mu(\D(z, r))}{\log r}&\ge
1 + \liminf_{n \to \infty}\Phi_n(\delta j_0, \ldots,\delta j_n),\\
\limsup_{r \to 0} \frac{\log\tilde  \mu(\D(z, r))}{\log r}&\ge 
1 + \limsup_{n \to \infty} \Psi_n(\delta j_0, \ldots,\delta j_n),
\end{aligned}
\end{equation}
for $\Phi, \Psi$ defined in \eqref{eq:PhiPsi}. Again, by the standard dimension estimates, \eqref{eq:log_mu>} shows 
\begin{align*}
\dim_H J \ge 1 + \inf_x\liminf_{n \to \infty}\Phi_n(x_0, \ldots, x_n),\\
\dim_P J \ge 1 + \inf_x\limsup_{n \to \infty}\Psi_n(x_0, \ldots, x_n)
\end{align*}
for $x = (x_0, x_1, \ldots) \in [\log a_N, \log b_N] \times [\log a_{N+1}, \log b_{N+1}] \times \cdots$. Together with \eqref{eq:dim}, this proves the lower estimate in Theorem~\ref{thm:B}.

\bibliography{slow}

\end{document}